\documentclass[10pt, letterpaper, notitlepage, oneside, fleqn]{article}

\usepackage{palatino,url}
\usepackage{graphicx}
\usepackage{amsmath}
\usepackage{amsfonts}
\usepackage{mathrsfs}
\usepackage{mathtools}
\usepackage{amsthm}
\usepackage{amssymb}
\usepackage[all]{xy}
\usepackage{verbatim}
\usepackage{fancyhdr}
\usepackage{color}
\usepackage{longtable}
\usepackage{bbm}
\usepackage{tikz}

\usepackage{ytableau}
\ytableausetup{mathmode, boxsize=2em}

\fancyheadoffset{0cm}

\usepackage[vcentermath]{youngtab}
\newcommand{\Y}[1]{{\tiny\yng(#1)}}

\newcommand{\s}{\sigma}

\newcommand{\y}{\lambda}

\newcommand{\g}{\gamma}
\renewcommand{\a}{\alpha}
\renewcommand{\b}{\beta}

\newcommand{\Ind}{\text{Ind}}

\newcommand{\Res}{\text{Res}}
\newcommand{\Hom}{\text{Hom}}

\newcommand{\End}{\text{End}}

\renewcommand{\dim}{\text{dim}}
\newcommand{\codim}{\text{codim}}

\newcommand{\FIW}{\text{FI}_{\mathcal{W}}}

\newcommand{\FI}{\text{FI}}

\newcommand{\cA}{\mathcal{A}}

\newcommand{\cM}{\mathcal{M}}

\newcommand{\cY}{\mathcal{Y}}

\newcommand{\bn}{{\bf n}}
\newcommand{\bm}{ {\bf m}}

\newcommand{\W}{\mathcal{W}}

\newcommand{\tY}{\mathbf{Y}}
\newcommand{\tH}{\mathbf{H}}

\newcommand{\oo}{\overline}

\newcommand{\C}{\mathbb{C}}

\newcommand{\Q}{\mathbb{Q}}
\newcommand{\R}{\mathbb{R}}
\newcommand{\Z}{\mathbb{Z}}
\newcommand{\F}{\mathbb{F}}

\renewcommand{\k}{\mathbbm{k}}

\newcommand{\Fr}{\text{Fr}}

\newtheorem{thm}{Theorem}[section]
\newtheorem{prop}[thm]{Proposition}
\newtheorem{lem}[thm]{Lemma}
\newtheorem{cor}[thm]{Corollary}

\theoremstyle{definition}
\newtheorem{defn}[thm]{Definition}
\newtheorem{rem}[thm]{Remark}
\newtheorem{example}[thm]{Example}
\newtheorem{problem}[thm]{Problem}
\newtheorem{notn}[thm]{Notation}

\theoremstyle{plain}

\date{\today}\date{\today}
\title{\bf Stability for hyperplane complements of type B/C and statistics on squarefree polynomials over finite fields}
\author{Rita Jim\'enez Rolland\footnote{The first author is grateful for the financial support from  PAPIIT-UNAM grant IA100816.}\  \  and Jennifer C. H. Wilson}

\begin{document}

\maketitle

\begin{abstract} 
In this paper we explore a relationship between the topology of the complex hyperplane complements $\cM_{BC_n}(\C)$ in type B/C and the combinatorics of certain spaces of degree--$n$ polynomials over a finite field $\F_q$. This relationship  is a consequence of the Grothendieck trace formula and work of Lehrer and Kim. We use it to prove a correspondence between a representation-theoretic convergence result on the cohomology algebras $H^*(\cM_{BC_n}(\C); \C)$, and an asymptotic stability result for certain \emph{polynomial} statistics on monic squarefree polynomials over $\F_q$ with nonzero constant term. This result is the type B/C analogue of a theorem due to Church, Ellenberg, and Farb in type A, and we include a new proof of their theorem. To establish these convergence results, we realize the sequences of cohomology algebras of the hyperplane complements as \emph{$\FIW$--algebras finitely generated  in $\FIW$--degree 2}, and we investigate the asymptotic behaviour of general families of algebras with this structure.  We prove a negative result implying that this structure alone is not sufficient to prove the necessary convergence conditions. Our proof of convergence for the cohomology algebras involves the combinatorics of their relators. 
\end{abstract}

\setcounter{tocdepth}{2}
\tableofcontents

\section{Introduction} 
\subsection*{Hyperplane complements and statistics on $\F_q[T]$}

Define complex hyperplane complements 
\begin{align*}
 &\cM_{{A}_n}(\C) = \{ (x_1, \ldots, x_n) \in \C^n \; | \; x_i \neq x_j  \text{ for all } i \neq j \} \\
 &\cM_{{BC}_n}(\C)= \{ (x_1, \ldots, x_n) \in \C^n \; | \; x_i \neq 0 \text{ for all } i, \text{ and }  x_i \neq \pm x_j  \text{ for all } i \neq j \}. 
\end{align*}

The Grothendieck--Lefschetz trace formula in $\ell$-adic cohomology and results of Lehrer \cite{LEHRER_LADIC} and Kim \cite{KIM} imply an amazing relationship between the complex cohomology of these hyperplane complements and point counts for certain spaces of polynomials over a finite field $\F_q$.  Church, Ellenberg, and Farb \cite{CEFPointCounting} describe this relationship for the hyperplane complement $\cM_{{A}_n}(\C)$, and use it to relate stability results for the cohomology groups $H^d(\cM_{{A}_n}(\C); \C)$ as symmetric group representations to stability results for \emph{polynomial} statistics on the set of monic squarefree degree--$n$ polynomials in $\F_q[T]$, in the limit as $n$ tends to infinity \cite[Theorem 3.7 and Theorem 1]{CEFPointCounting}. 

One goal of this paper is to prove the analogues of Church, Ellenberg, and Farb's results in type B/C. We investigate the relationship between the complex cohomology of the hyperplane complement $\cM_{{BC}_n}(\C)$ and statistics on the set  $\cY_n(\F_q)$  of monic squarefree degree--$n$ polynomials in $\F_q[T]$ with nonzero constant term. The space $\cM_{{BC}_n}(\C)$ has an action of the hyperoctahedral group $B_n$ (Definition \ref{DefnBn}), and the following theorem --  a twisted version of the Grothendieck--Lefschetz trace formula, specialized to the scheme $\cY_n$ -- relates the structure of the cohomology groups $H^d(\cM_{{BC}_n}(\C);\C)$ as $B_n$--representations to point-counts on $\cY_n(\F_q)$.\\[-.5em]

\noindent {\bf Theorem \ref{COUNTING}  (A point-counting formula for $\mathcal{Y}_n(\mathbb{F}_q)$). } {\it Let $q$ be an integral power of a prime number $p>2$ and let $\chi$ be a class function on $B_n$. Then for each $n\geq 1$ we have
	\begin{align*}
	&& \sum_{f\in \mathcal{Y}_n(\mathbb{F}_q)} \chi(f)=\sum_{d=0}^n  (-1)^{d}q^{n-d} \big\langle\chi, H^{d}({\cM_{BC_n}} (\mathbb{C});\mathbb{C}\big)\big\rangle_{B_n}. && \text{ \qquad \qquad \qquad \qquad \qquad \qquad (\ref{COUNT})}
	\end{align*} } 
\noindent Section \ref{SectionPoint-CountingFormula} describes how to interpret a $B_n$ class function $\chi$ as a function on polynomials in $\mathcal{Y}_n(\mathbb{F}_q)$, by considering the action of the Frobenius morphism on an associated set  $\mathcal{Y}_n(\overline{\F}_q)$. 


In previous work the second author showed that there is a sense in which the $B_n$--representations $H^d(\cM_{{BC}_n}(\C);\C)$ stabilize as $n$ grows \cite[Theorem 5.8]{FIW2}, using a description of these cohomology groups due to Orlik and Solomon \cite{OrlikSolomon}. In Theorem \ref{ASYMCOUNTING} we use this \emph{representation stability} result for the complex cohomology groups $H^d(\cM_{{BC}_n}(\C);\C)$ and a combinatorial result, Theorem \ref{TheoremStabilityHyperplanes},  to prove asymptotic stability for certain \emph{polynomial} statistics on the polynomials $\cY_n(\F_q)$. \\[-.5em]

\noindent {\bf Theorem \ref{ASYMCOUNTING}   (Stability for polynomial statistics on $\mathcal{Y}_n(\mathbb{F}_q)$). } {\it  Let $q$ be an integral power of an odd prime. For any hyperoctahedral character polynomial  $P \in \Q[X_1, Y_1, X_2, Y_2, \ldots]$ the normalized statistic $q^{-n}\sum_{f\in \mathcal{Y}_n(\mathbb{F}_q)} P(f)$ converges to a limit as $n\rightarrow \infty$. In fact
		\begin{align*}
	&& \lim_{n\to \infty} q^{-n}\sum_{f\in \mathcal{Y}_n(\mathbb{F}_q)} P(f)=\sum_{d=0}^\infty  \frac{ \lim_{m\to\infty}\big\langle P\vert_{B_m}, H^{d}({\cM_{{BC}_m}(\C)},\C)\big\rangle_{B_m}}{(-q)^{d}}
	&&  \text{ \qquad \qquad \qquad \quad  (\ref{ASYMEQ})} \end{align*}
	and the series in the right hand side converges.\\[-.5em] } 

\noindent The functions $X_r$ and $Y_r$ are  \emph{signed-cycle-counting} class functions on $B_n$ (Definition \ref{DefnHypCharPoly}). Given a polynomial $P$ in these functions, the values $P(f)$ encode data on the irreducible factors of the polynomial $f$ and the square roots of their zeroes in $\overline{\F}_q$, as we summarize on the upcoming page. Precise descriptions of character polynomials are given as $B_n$ class functions in Section \ref{REP} and as statistics on $ \mathcal{Y}_n(\mathbb{F}_q)$ in Section \ref{SectionPoint-CountingFormula}.

\subsection*{Convergence and nonconvergence results: algebras in degree $2$}

To prove Theorem \ref{ASYMCOUNTING} we study the algebraic structure of the cohomology rings of the family $\{ \cM_{{BC}_n}(\C)\}_n$.  The proof is combinatorial, and the approach was motivated by our general interest in understanding what are the combinatorial features of the generators and relations of a sequence of algebras that allow for convergence results of the form of \cite[Theorem 1]{CEFPointCounting} and  our Theorem \ref{ASYMCOUNTING}. The  cohomology rings of both the families $\{\cM_{{A}_n}(\C)\}_n$ and  $\{ \cM_{{BC}_n}(\C)\}_n$ have the structure of  \emph{$\FIW$--algebras finitely generated in $\FIW$-degree $\leq 2$}, and a second goal of this paper is to investigate the asymptotic properties implied by this structure. 

Let $\W_n$ generically denote either the family of symmetric groups $S_n$ or the family of hyperoctahedral groups $B_n$.  In Section \ref{FIW}, we review the definitions of $\FIW$--modules and $\FIW$--algebras, and the algebraic framework they provide for studying sequences of $\W_n$--representations. In earlier work  \cite{JIMWIL1}, the authors prove that any $\FIW$--algebra finitely generated in $\FIW$--degree 0 or 1 satisfies the desired convergence condition; this result is stated in Theorem \ref{ConvergenceCriterion}. In particular, consider the sequence of commutative or graded-commutative algebras $A^*_n = \Q[x_1, x_2, \ldots, x_n]$ with an action of $\W_n$ by permuting the subscripts and possibly negating the variables. Then for any integer $q>1$ and any type $\W$ character polynomial $P$, the following limit converges: 
\begin{align} \label{DefiningConvergence} \qquad  \qquad \qquad \qquad \qquad \qquad \sum_{d=0}^\infty  \frac{ \lim_{m\to\infty}\big\langle P\vert_{\W_m}, A^d_m \big\rangle_{\W_m}}{(-q)^{d}}
\end{align}
This same convergence result holds if we take the sequence of (graded)-commutative algebras in finitely many collections of variables 
$$A^*_n = \Q[y^{(1)}, y^{(2)}, \ldots, y^{(a)}, x^{(1)}_1, x^{(1)}_2, \ldots, x^{(1)}_n, x^{(2)}_1, x^{(2)}_2, \ldots, x^{(2)}_n, \ldots, x^{(c)}_1, x^{(c)}_2, \ldots, x^{(c)}_n]$$ 
with an action of $\W_n$ by permuting the subscripts and possibly negating the variables. 

A natural question, then, is whether convergence is again automatic for for $\FIW$--algebras finitely generated in $\FIW$-degree $2$. Specifically, suppose $A^*_n$ is the sequence of polynomial algebras $$ A^*_n = \Q[\;  x_{i,j} \; | \; i,j = 1, \ldots, n \; ]$$ with an action of $\W_n$ by permuting the indices and possibly negating the variables. We may assume the variables satisfy any one of $x_{i,j} \neq x_{j,i}, x_{i,j} =  x_{j,i}$, or $x_{i,j} =-  x_{j,i}$. Then the question is,  will this sequence of algebras be convergent in the sense of Equation (\ref{DefiningConvergence})? 
 In Theorem \ref{NOLIMIT}, we show that this is not the case; it fails even for the trivial characters $P=1$.  \\[-.5em]
 
\noindent {\bf Theorem \ref{NOLIMIT}. (Nonconvergence for symmetric $\FIW$--algebras generated in $\FIW$--degree 2).} {\it Let $\k$ be a subfield of $\C$, and $q>1$ an integer. 
 Let $V$ be a graded $\FIW$--module over $\k$ supported in positive grades containing a free $\FIW$--module on a representation of $\W_2$.  Let $A^*$ be an $\FIW$--algebra  containing the free symmetric algebra on $V$. Then there exist character polynomials $P$ for which the following series does not converge:  $$ \sum_{d=0}  \frac{ \lim_{n \to \infty} \langle P_n, A^d_n \rangle_{\W_n}}{q^d}. $$ }

Theorem \ref{NOLIMIT} shows that the convergent results of \cite[Theorem 1]{CEFPointCounting} and Theorem \ref{ASYMCOUNTING} do not follow formally from the combinatorics of the generators of these algebras. In Proposition \ref{PropArnoldAlg} we give a new proof of \cite[Theorem 1]{CEFPointCounting}, and we extend this strategy to prove Theorem \ref{ASYMCOUNTING}. These proofs examine how combinatorial aspects of the {\it relations} of these algebras can drive their convergence behaviour. 

\subsection*{Statistics on squarefree polynomials and hyperplane complements of type B/C}

Recall that $\cY_n(\F_q)$  denotes the set of monic squarefree polynomials in $\F_q[T]$ with nonzero constant term. It is natural to ask about the distribution of irreducible degree--$r$ factors of these polynomials. Roots of any degree--$r$ irreducible factor lie in $\mathbb{F}_{q^r}$, and it is a more subtle question to ask about the nature of the square roots of these roots. These data are encoded by the action of the Frobenius morphism on the set  $\cY_n(\overline{\F}_q)$, and allows us to interpret the hyperoctahedral signed-cycle-counting class functions $X_r$ and $Y_r$ (Definition \ref{DefnHypCharPoly}) as the following functions on $\cY_n(\F_q)$. 
\begin{align*}
	&X_r(f)=\text{\# degree--$r$ irreducible factors  of $f$ whose roots are {\it quadratic residues} over $\F_{q^r}$} \\
	&Y_r(f)=\text{\# degree--$r$ irreducible factors  of $f$ whose roots are {\it quadratic nonresidues} over $\F_{q^r}$} \\ 	
	& X_r(f)+Y_r(f)=\text{total \# degree--$r$ irreducible factors  of $f$}. 
\end{align*}
	
Details are given in Sections \ref{MBCSchemes} and \ref{SectionPoint-CountingFormula}. With these definitions, Theorem \ref{ASYMCOUNTING} then states that for $q$ an odd prime power, given any polynomial $P$ in the class functions $X_r$ and $Y_r$, the limits in Formula (\ref{ASYMEQ}) exist and are equal. It implies in particular that the expected value of $P$ on $ \mathcal{Y}_n(\mathbb{F}_q)$ converges. \\[-.5em]
	
\noindent {\bf Corollary \ref{CorExpectedValue}. (Stability for the expected value of  polynomial statistics on $\mathcal{Y}_n(\mathbb{F}_q)$).}  {\it Let $q$ be an integral power of an odd prime. For any polynomial  $P \in \Q[X_1, Y_1, X_2, Y_2, \ldots]$ the expected value of $P$ on $\mathcal{Y}_n(\mathbb{F}_q)$ converges as $n$ tends to infinity, and its limit is
\begin{align*}
	&&\lim_{n\to \infty} \frac{\sum_{f\in \mathcal{Y}_n(\mathbb{F}_q)} P(f)}{|\cY(F_q)|} = \left(\frac{q+1}{q-1}\right) \sum_{d=0}^\infty  \frac{ \lim_{m\to\infty}\big\langle P_m, H^{d}({\cM_{{BC}_m}(\C)},\C)\big\rangle_{B_m}}{(-q)^{d}}.
	\end{align*} }


\subsubsection*{Sample computations}

To illustrate these results,  in Section \ref{EXPLICIT} we evaluate Formula (\ref{ASYMEQ}) in some specific examples. We first review a result of Douglass \cite[Formula (1.1)]{DouglassArrangement} on a decomposition of the cohomology groups $H^d({\cM_{BC}}_n(\C); \C)$ as $B_n$--representations. Then, using results of Brieskorn \cite[Th\'eor\`eme 7]{Brieskorn} and Douglass, we compute the stable inner products on the right-hand side of Formula \ref{ASYMEQ} for the character polynomials $P=1$, $P=X_1 - Y_1$ (Lemma \ref{X-Y}) and $P=X_1 + Y_1$ (Lemma \ref{INNER}). We then show that, from these computations, we can deduce the following (well-known) statistics on the polynomials over $\F_q[T]$:

\begin{itemize}

\item {\bf (Proposition \ref{NUMBER}). } The number of degree--$n$ monic squarefree polynomials $f$ in $\mathbb{F}_q[T]$ with $f(0)\neq0$ is 
$$|\mathcal{Y}_n(\mathbb{F}_q)|=q^n-2q^{n-1}+2q^{n-2}-\ldots+ (-1)^{n-1}2q+(-1)^n.$$

\item  {\bf (Proposition \ref{NUMLINEAR}). } The expected number of linear factors in a random degree--$n$ monic squarefree polynomial $f$ in $\mathbb{F}_q[T]$ with $f(0)\neq0$ converges to 
$$  \frac{q-1}{q+1} \qquad \text{ in the limit as $n\longrightarrow\infty$.}$$ 

\item  {\bf (Proposition \ref{NUMQR}). } The expected number of roots that are quadratic residues over $\F_q$ for a random degree--$n$ monic squarefree polynomial $f$ in $\mathbb{F}_q[T]$ with $f(0)\neq0$
converges to 
$$ \frac{q-1}{2(q+1)} \qquad \text{ in the limit   as  $n\longrightarrow\infty$.}$$

\end{itemize}

These three statistics can be computed directly by combinatorial methods on $\F_q[T]$, for example, they are a special case of a computation done by Alegre, Juarez, and Pajela \cite[Theorem 20]{WEIYANSTUDENTS} using the generating function techniques of Fulman  \cite{FULMAN}. We have taken this unconventional approach to computing these statistics in  Section \ref{EXPLICIT} in order to showcase the extraordinary result of Grothendieck, Lehrer, Kim, and others that we can extract combinatorial data on polynomials over finite fields from topological properties of complex hyperplane complements, and vice versa.

\subsection*{ Related work}

\subsubsection*{New work of Matei}
After posting a preprint of this paper, we learned of recent work of Vlad Matei \cite{Matei}. Matei independently proves Theorem  \ref{COUNTING} [Theorem 4, Matei], also using methods inspired by Church, Ellenberg, and Farb's proof of their Theorem 3.7 \cite{CEFPointCounting}.  He combines this result with a description of the cohomology groups $H^d(\cM_{BC_n}(\C), \C)$ due to Henderson \cite{HendersonHyperplanes} to prove a theorem \cite[Theorem 1]{Matei} on the distribution of monic degree--$n$ polynomials of the form 
$$ f(T) = \Big(g(T)\Big)^2 + T \Big(h(T)\Big)^2, \qquad g(T), h(T) \in \F_q[T].$$

\subsubsection*{New work of Casto}
While finishing this paper,  we learned of new work of  Casto \cite{CASTO} establishing a general asymptotic stability result for statistics on the  $i$th roots of the zeroes of polynomials over $\F_q$, and the associated orbit configuration spaces. Casto \cite{CASTO} generalizes techniques of Farb and Wolfson  \cite{FARBWOLFSON} to \emph{FI$_G$--modules}. He proves asymptotic stability by topological methods, showing that orbit configuration spaces have {\it convergent cohomology} \cite[Theorem 1.3 \& Section 3]{CASTO}. His results \cite[Theorem 1.4 and 1.5]{CASTO}  recover Theorems \ref{COUNTING} and \ref{ASYMCOUNTING} as the special case $i=2$.

\subsubsection*{Other related work} 

Convergence of the left hand side of Formula (\ref{ASYMEQ}) in Theorem \ref{ASYMCOUNTING} above could also be proven using the generating function techniques that have been employed in recent work such as Fulman \cite{FULMAN} and Chen \cite{WEIYAN}; see  \cite[Corollay 4 (b)]{WEIYAN}. These methods can also be used to obtain additional results on the stable values, for example, Chen \cite{WEIYAN} shows that the stable Betti numbers of unordered configuration spaces are \emph{quasipolynomial} and satisfy linear recurrence relations. It should be possible to adapt these techniques for the type B/C analogues. 

The use of point-counting over finite fields to obtain topological information about complex reflection group arrangements has appeared in the work of Kisin and Lehrer \cite{KISIN-LEHRER}. See also Chen \cite{WEIYAN} and  Fulman--Jim\'enez Rolland--Wilson \cite{FJRW} for related work on the spaces of maximal tori in Lie groups of types A and B/ C, respectively.

Farb--Wolfson \cite[Theorem B]{FARBWOLFSON} prove {\it \'etale homological} and {\it representation stability} for the cohomology of configuration spaces of $n$ points on smooth varieties. They use the Grothendieck--Leftschetz formula to relate this result to point counts over finite fields, and they prove asymptotic stability for these point counts by establishing subexponential bounds on the growth of the unstable cohomology of those spaces  \cite[Theorems A and C]{FARBWOLFSON}.

Gadish \cite[Theorem A]{GADISH} derives a version of the Grothendieck trace formula for ramified covers that is suited to applications in representation stability. With this formula he performs explicit computations for the Vieta cover of the variety of polynomials, and describes factorization statistics of polynomials over finite fields \cite[Section 3]{GADISH}.

\subsection*{Acknowledgments} We are grateful to Tom Church, Weiyan Chen, Kevin Casto, Sean Howe, and L\'aszl\'o Lov\'asz for helpful conversations. We thank Benson Farb for suggesting this project to us.  

\section{Foundations: $\FIW$-modules and character polynomials}\label{FIW}

Church--Ellenberg--Farb \cite{CEF} introduced the theory of \emph{FI--modules} as an algebraic framework for studying sequences of symmetric group representations. Their results were generalized by the second author to sequences of representations of the classical Weyl groups in type B/C and D \cite{FIW1, FIW2}. In this section we summarize the relevant terminology and results.

\begin{defn} \label{DefnBn}
Let $B_n$ denote the Weyl group in type $B_n/C_n$, which we call the the \emph{hyperoctahedral group} or \emph{signed permutation group}. We define $B_n$ as the subgroup of permutations $S_{\Omega}\cong S_{2n}$ on the set $\Omega = \{ 1, \oo1, 2, \oo2, \ldots, n, \oo{n}\}$ given by $$ B_n = \{ \sigma \in S_{\Omega} \; \; | \; \; \sigma(\oo{a}) = \oo{\sigma(a) }\quad \text{for all $1 \leq a \leq n$} \; \}, \qquad B_0 = \{1\}.$$  There is a natural surjection $B_n \to S_n$ by forgetting signs of the elements of $\Omega$. 
\end{defn}

Throughout the paper we will use $\W_n$ to generically denote either Weyl group family: the symmetric groups $S_n$ in type A or the signed permutation groups $B_n$ in type B/C.

\subsection{The theory of $\FIW$--modules}
 
\begin{defn} {\bf (The category $\FIW$; $\FIW$--modules.)} Let $\W_n$ be one of the families $S_n$ or $B_n$. To each of these families we associate a category $\FIW$, defined as follows. Its objects are finite sets $${\bf 0} = \varnothing \qquad \text{and} \qquad \bn = \{1, \oo1, \ldots, n, \oo{n}\} \qquad \text{for $n \in \Z_{\geq 1}$},$$ where $\oo{a}$ is shorthand for $(-a)$ for any integer $a$. The $\FIW$ morphisms are generated by its endomorphisms $\End(\bn) = \W_n$ and the canonical inclusions $I_n : \bn \rightarrow (\bn + \mathbf{1})$. An \emph{$\FIW$--module} $V$ over a commutative, unital ring $R$ is a functor from $\FIW$ to the category of $R$--modules; its image is sequence of $\W_n$--representations $V_n:=V(\bn)$ with actions of the $\FIW$ morphisms. 

 In type A, the category $\FI_A$ is equivalent to the category $\FI$ studied by Church--Ellenberg--Farb \cite{CEF}. In type $B/C$ we denote the category by $\FI_{BC}$. 	
 \end{defn}

\begin{defn} {\bf (Graded $\FIW$--modules; graded $\FIW$--algebras.)}  A \emph{graded $\FIW$--module $V^*$} over a ring $R$ is a functor from $\FIW$ to the category of graded $R$--modules. A \emph{graded $\FIW$--algebra} over a ring $R$ is a functor from $\FIW$ to the category of graded $R$--algebras.  We will refer to $d$ as the \emph{graded--degree} and $n$ as the \emph{$\FIW$--degree} of the $R$--module $V^d_n$.
\end{defn}
	
\begin{defn} {\bf (Finite generation; degree of generation; finite type.)}
	An $\FIW$--module $V$ is \emph{generated (as a module)} by elements $\{v_i\} \subseteq \coprod_n V_n$ if $V$ is the smallest $\FIW$--submodule of $V$ containing $\{v_i\}$.  We call  $\{v_i\}$ an \emph{(additive) generating set} for $V$. A graded $\FIW$--algebra $V=V^*$ is \emph{generated (as an algebra)} by elements $\{v_i\} \subseteq \coprod_n V_n$ if $V$ is the smallest  $\FIW$--subalgebra containing the elements $\{v_i\}$.  We call  $\{v_i\}$ an \emph{(algebra) generating set} for $V$. An $\FIW$--module $V$ (respectively, a graded $\FIW$--algebra $V=V^*$)is \emph{finitely generated} as a module (respectively, as an algebra) if it has a finite generating set, and $V$ is \emph{generated in degree $\leq m$} if it is generated by $\coprod_{k=0}^m V_k$.  A graded $\FIW$--module or algebra $V^*$ has \emph{finite type} if each graded piece $V^d$ is finitely generated as an $\FIW$--module.
\end{defn}

\begin{defn} {\bf (The representable $\FIW$--modules $M_\W(\bm)$.)}  Following the notation introduced by Church--Ellenberg--Farb \cite{CEF}, we write $$M_\W(\bm):= R\big[\Hom_{\FIW}(\bm, -)\big]$$ for the represented $\FIW$--module over $R$ generated in degree $m$. An orbit-stabilizer argument shows that for each $n$ there is an isomorphism of $\W_n$--representations $$ M_\W(\bm)_n \cong R\big[\W_n/\W_{n-m} \big].$$ We denote represented $\FI$--modules by  $M(\bm)$ and represented $\FI_{BC}$--modules by $M_{BC}(\bm)$.
	 \end{defn}
	 	 
	 An $\FIW$--module over $R$ is finitely generated in degree $\leq k$ if and only if it is a quotient of a finite direct sum of represented functors $M_\W(\bm)$ with $m \leq k$ \cite[Proposition 2.3.5]{CEF}; \cite[Proposition 3.15]{FIW1}. 
		 
\begin{defn} {\bf (The free $\FIW$-modules $M_{\W}(U)$.)} Fix a nonnegative integer $m$ and a $\W_m$--representation $U$. Define the $\FIW$--module  $$M_\W (U) := M_\W(\bm) \otimes_{R[\W_m]} U$$ using the right action of $\W_m$ on $M_{\W}(\bm)=   R\big[\Hom_{\FIW}(\bm, \bn)].$ We call $M_{\W}(U)$ the \emph{free $\FIW$--module on $U$}.  As $\W_n$--representations,  $$M_{\W}(U)_n \cong \Ind_{\W_m \times \W_{n-m}}^{\W_n} U \boxtimes R, \qquad \text{where $R$ carries a trivial $\W_{n-m}$--action.} $$ Over fields of characteristic zero the decomposition of $M_{\W}(U)_n$  into irreducible representations is described by the branching rules. If $U$ is an irreducible $S_d$--representation associated to a partition $\y$ of $d$, then we may write $M(\y)$ in place of $M(U)$. Similarly, for $W$ the irreducible $B_d$--representation associated to double partition $(\y, \mu)$ we write $M_{BC}(\y, \mu)$ for $M_{BC}(W)$.\medskip
\end{defn}

\subsection{The representation theory of $\W_n$}\label{REP}

	Just as the irreducible complex representations of $S_n$ are in natural bijection with \emph{partitions} of $n$, the irreducible complex representations of  $B_n$ are in  bijection with \emph{double partitions} of $n$, that is, ordered pairs of partitions $(\lambda, \mu)$ such that $|\lambda|+|\mu|=n$. We denote the corresponding representation by $V_{(\lambda,\mu)}$. See (for example) Geck--Pfeiffer \cite{GeckPfeiffer} for a detailed development of the representation theory of these Weyl groups, or (for example) Wilson \cite[Section 2.1]{FIW1} for a summary suited to present purposes. Each complex irreducible representation of these groups is defined over the rational numbers, or any subfield of $\C$.  For the remainder of this section we will take coefficients to be in a subfield $\k$ of $\C$.

\begin{defn} {\bf (Character polynomials for $S_n$).} For each $r \in \Z_{\geq 1}$, MacDonald \cite[I.7 Example 14]{MacdonaldSymmetric} defined a  class function $X_r: \coprod_n S_n \to \Z$ that take a permutation $\sigma$ to the number $X_r(\sigma)$ of $r$--cycles in its cycle type.  A \emph{character polynomial for $S_n$} with coefficients in a ring $\k$ is a polynomial in these class functions $P \in \k[X_1, X_2, \ldots]$. The \emph{degree} of a character polynomial is defined by assigning $\deg(X_r) = r$ for $r\geq 1$. A character polynomial $P$ defines a class function on $S_n$ for all $n \geq 0$; we denote its restriction to $S_n$ by $P_n$. 
\end{defn}
\begin{defn} {\bf (Signed cycle type of a signed permutation).} \label{DefnSignedCycleType} The conjugacy classes of $B_n$ are classified by \emph{signed cycle type} as follows. For $r \in \Z_{\geq 1}$, a signed permutation in $B_n$ is called an \emph{$r$--cycle} if it projects to an $r$--cycle in $S_n$.  A \emph{positive $r$-cycle} is an $r$--cycle that negates an even number of letters in $\Omega = \{ 1, \oo1, 2, \oo2, \ldots, n, \oo{n}\}$; these are elements of the form $$(a_1 a_2 \cdots  a_r)(\oo{a_1} \, \oo{a_2} \cdots \oo{a_r}) \in S_{\Omega}, \qquad a_i \in \Omega$$ when expressed in cycle notation as a permutation on  ${\Omega}.$  A positive $r$--cycle has order $r$. A \emph{negative $r$--cycle} is a signed permutation that negates an odd number of letters; these have the form $$(a_1 a_2 \cdots a_r \oo{a_1} \cdots \oo{a_r})  \in S_{\Omega}, \qquad a_i \in \Omega.$$ A negative $r$--cycle has order $2r$, and its  $r^{th}$ power is a product of transpositions $$(a_1 \; \oo{a_1}) (a_2 \; \oo{a_2}) \cdots (a_r \; \oo{a_r}).$$
\end{defn}
Young  \cite{YoungHyperoctahedral} proved that signed permutations factor uniquely as a product of positive and negative cycles, and two signed permutations are conjugate if and only if they have the same signed cycle type. 
\begin{defn} {\bf (Character polynomials for $B_n$).}  \label{DefnHypCharPoly}
Given  $\sigma \in B_n$, let $X_r(\sigma)$ be the number of positive $r$--cycles in its signed cycle type, and $Y_r(\sigma)$ the number of negative $r$--cycles. A \emph{hyperoctahedral character polynomial} with coefficients in a ring $\k$ is a polynomial $P \in \k[X_1, Y_1, X_2, Y_2, \ldots]$. Each character polynomial defines a class function on $\coprod_n B_n$;  we write $P_n$ to mean its restriction to $B_n$. We define the \emph{degree} of a character  polynomial by setting $\deg(X_r)=\deg(Y_r)=r$.
\end{defn}

\begin{notn}{\bf (The inner product $\langle - , - \rangle_{G}$).} For a finite group $G$, we write $\langle - , - \rangle_{G}$ for the standard inner product on the $\C$-valued class functions on $G$. By abuse of notation we may write either class functions or $G$--representations in its argument; a representations should be taken to represent the corresponding character. 
\end{notn}

\subsection{Finitely generated $\FIW$--modules are representation stable}
A central result of the work of Church--Ellenberg--Farb \cite{CEF} and Wilson \cite{FIW1, FIW2} are constraints on the structure of finitely generated $\FIW$--modules.  Specifically, a finitely generated $\FIW$--module over characteristic zero is \emph{uniformly representation stable} in the sense of Church--Farb \cite[Definition 2.3]{CF}, and its characters are \emph{eventually polynomial}, in the sense of the following theorem. 

\begin{thm} \label{REPSTABILITY}{\bf (Constraints on finitely generated $\FIW$--modules).} Let $V$ be an $\FIW$--module over a subfield $\k$ of $\C$ which is finitely generated in degree $\leq m$. 
	\begin{itemize}
\item {\bf (Uniform representation stability)} \cite[Prop. 3.3.3]{CEF}; \cite[Theorem 4.27]{FIW1}  \\ The sequence $V_n$ is uniformly representation stable with respect to the maps induced by the $\FIW$--morphisms the natural inclusions $I_n: \bn \to ({\bf n+1})$, stabilizing once $n$ is at least $m +\max($generation degree of $V$, relation degree of $V$). 
\item {\bf (Character polynomials)} \cite[Theorem 3.3.4]{CEF}; \cite[Theorem 4.6]{FIW2} \\ Let $\chi_n$ denote the character of the $\W_n$--representation $V_n$. Then there exists a unique character	polynomial $F_V$  of degree at most $m$ such that $F_V (\s) = \chi_n(\s)$ for all $\s \in \W_n$ and $n >>0.$
\end{itemize}
\end{thm}

\noindent Relation degree is defined in \cite[Definition 3.18]{FIW1}. Church--Ellenberg--Farb and Wilson proved that finitely generated $\FIW$--modules are Noetherian  \cite[Theorem 1.3]{CEF}; \cite[Theorem 4.21]{FIW1}, and so a finitely generated $\FIW$--module $V$ necessarily has finite relation degree.

\subsection{Existing asymptotic results}
 Below is a summary of results on asymptotics of character polynomials, which we will use to prove the results in Section \ref{ASYMSection}.

 \begin{defn}{\bf (Asymptotic equivalence; asymptotic bounds; Big and little O notation).} 
 Let $f, g \colon \Z_{\geq 0} \to  \R$ be functions. We say that $f$ is \emph{asymptotically equivalent} to $g$ and write $f \sim g$  if $$\lim_{d \to \infty} \frac{f(d)}{g(d)}=1.$$  
We say  $f$ is \emph{asymptotically dominated} by $g$ and write $f \lesssim g$ if  $$ \limsup_{d \to \infty}\frac{f(d)}{g(d)} \leq 1.$$ 
The function $f$ \emph{is order $O(g)$} or \emph{asymptotically bounded above} by $g$ if $$\limsup_{d \to \infty} \left\vert\frac{f(d)}{g(d)}\right\vert < \infty,$$ equivalently, if $|f(d)| \leq C|g(d)|$ for some constant $C$ and all $d$ sufficiently large. 
The function $f$ \emph{is order $o(g)$} if $$\lim_{d \to \infty} \frac{f(d)}{g(d)}=0.$$ 
\end{defn}

\begin{prop}{\bf (The inner product of character polynomials stabilizes).} \cite[Proposition 3.9]{CEFPointCounting}; \cite[Proposition 3.1]{JIMWIL1}   \label{STABPOL} Let $\k$ be a subfield of $\C$, and let $\W_n$ represent one of the families $S_n$ or $B_n$. Let $P, Q $ be two  character polynomials for $\W_n$. Then the inner product $\langle Q_n, P_n \rangle_{\W_n}$  is independent of $n$ for $n \geq \deg(P) + \deg(Q)$. 
\end{prop}

\begin{lem} \label{FGCHARALGSTAB} {\bf (Stability for characters of finitely generated $\FIW$--algebras).} \cite[Lemma 3.3]{JIMWIL1} \\ Let $\k$ be a subfield of $\C$, and let $\W_n$ represent one of the families $S_n$ or $B_n$. Suppose that $A^*$ is an associative $\FIW$--algebra over $\k$ that is generated as an $\FIW$--algebra by finitely many elements of positive graded-degree. Then for each $d$ and any $\W_n$ character polynomial $P$, the following limit exists: $$  \lim_{n \to \infty} \langle P_n, A^d_n \rangle_{\W_n}.$$

\end{lem}

\begin{lem} {\bf (Bounding coinvariants).} \cite[Definition 3.12]{CEFPointCounting}; \cite[Lemma 3.4]{JIMWIL1} \label{EQUIV} Let $A^d$ be the $d^{th}$ graded piece of a graded $\FIW$--module over a subfield $\k$ of $\C$. For a function $g: \Z_{\leq 0} \to \R$, the following are equivalent: 
	\begin{enumerate}
		\item[I.] For each $a \geq 0$ there is a function $F_a(d)$ that independent of $n$ and order $O(g)$ such that $$ \dim_{\k} \Big( (A^d_n)^{\W_{n-a}} \Big) \leq F_a(d) \qquad \text{for all $d$ and $n$}. $$
		\item[II.] For each  $\W_n$ character polynomial $P$ there is a function $F_P(d)$ that is independent of $n$ and order $O(g)$ such that $$ | \langle P_n,  A^d_n \rangle_{B_n} | \leq F_P(d) \qquad \text{for all $d$ and $n$}. $$
	\end{enumerate}	
\end{lem}

\section{Convergence and nonconvergence results}\label{ASYMSection}

In recent work \cite{JIMWIL1} the authors prove that $\FIW$--algebras finitely generated in $\FIW$--degree zero or one satisfy a certain convergence result; the precise statement is given in Theorem \ref{ConvergenceCriterion} below. We use this theorem to investigate the structure of coinvariant algebras of type A and type B/C. Our results can be interpreted as asymptotic convergence results for 'polynomial statistics' on maximal tori in the corresponding matrix groups over finite fields; see \cite[Theorem 5.6]{CEFPointCounting} and \cite[Theorem 4.3]{JIMWIL1}. 

Many naturally arising $\FIW$--algebras, however, are finitely generated by elements in $\FIW$-degree two or higher, and 
so do not fall within the scope of Theorem \ref{ConvergenceCriterion}. These include the main examples of Sections \ref{ASYMSection} and \ref{POLCOUNTING}, the cohomology algebras of the hyperplane complements associated to braid arrangements in type A and B/C. Again we are faced with the question of whether these $\FIW$--algebras are convergent in the sense of Theorem \ref{ConvergenceCriterion}: in Section \ref{POLCOUNTING} we describe how this result corresponds to convergence results for certain statistics on polynomials over finite fields. 

An optimistic, if naive, conjecture is that all finitely generated $\FIW$--algebras satisfy the form of convergence of Theorem \ref{ConvergenceCriterion}.  Unfortunately, in Section \ref{SectionConvergenceFails}, we show that this is not the case in general. In Sections \ref{SectionBraid} and \ref{SectionBCBraid}, however, we develop combinatorial strategies for proving convergence in our specific examples, the cohomology of the hyperplane complements. Proposition \ref{PropArnoldAlg} gives a new proof of convergence in type A, and Theorem \ref{TheoremStabilityHyperplanes} establishes convergence in type B/C.

\subsection{Failure of convergence for free $\FIW$--algebras generated in degree 2} \label{SectionConvergenceFails}

In earlier work \cite[Theorem 3.5]{JIMWIL1}  the authors prove a convergence result for free (graded)-commutative $\FIW$--algebras on generators in $\FIW$--degrees zero or one. 

\begin{thm}{\bf (Criteria for convergent $\FIW$--algebras).} {\cite[Theorem 3.5]{JIMWIL1}}. \label{ConvergenceCriterion} Let $\W_n$ be one of the families $S_n$ or $B_n$, and let $\k$ be a subfield of $\C$. For nonnegative integers $b,c$, define a graded $\FIW$--module over $\k$
$$V \cong M_{\W}({\bf 0})^{\oplus b} \oplus M_{\W}({\bf 1})^{\oplus c}$$ 
with positive gradings. 
Let $\Gamma^*$ be the commutative, exterior, or graded-commutative algebra generated by $V$. Let $\{A_d^n\}$ be any sequence of graded $\W_n$-representations such that $A^d_n$ is a subrepresentation of  $\Gamma_{\W_n}^d$. Then for any $\W_n$ character polynomial $P$ and integer $q >1$, the following sequence converges asbolutely: $$ \sum_{d=0}^{\infty} \frac{\lim_{n \to \infty} \langle P_n, A^d_n \rangle_{\W_n}}{q^d} . $$ 
\end{thm}

In this next proposition we show that if we replace $V$ by $$M(\,\Y{2}\,) = M(\mathbf{2}) \otimes_{\k[S_2]} \k \text{\ \ \ \ or \ \ \ \  }M_{BC}(\,\Y{2}\, ,\varnothing)= M_{BC}(\mathbf{2}) \otimes_{\k[B_2]} \k,$$  
which are in some sense the next 'smallest' projective  $\FIW$--modules,  then the convergence result no longer holds.  

\begin{prop} {\bf (Nonconvergence for symmetric $\FIW$--algebras on $ M(\,\Y{2}\,)$  or $ M_{BC}(\,\Y{2}\, , \varnothing)$).} \label{FIWAlgebrasOnTrv2} Let $\k$ be a subfield of $\C$, and $q>1$ an integer. 
  Let $V$ be the free $\FIW$--module on the trivial $\W_2$--representation $\k$, concretely, this is $ M(\,\Y{2}\,)$  in type $A$ or $ M_{BC}(\,\Y{2}\, , \varnothing)$ in type B/C. Assume $V$ is graded by some positive grading.  Let $A^*$ be an $\FIW$--algebra containing the symmetric algebra on $V$. Then there exist character polynomials $P$ for which the following series does not converge:   $$ \sum_{d=0}  \frac{ \lim_{n \to \infty} \langle P_n, A^d_n \rangle_{\W_n}}{q^d}. $$ 
			\end{prop}

\begin{proof} Suppose that $V$ is concentrated in graded--degree $\ell \geq 1$. It is enough to consider the character polynomial $P=1$ and the case that $A^*$ is equal to the symmetric $\FIW$--algebra freely generated by $V$. The inner product  $\langle 1, A^d_n \rangle_{\W_n}$ is the dimension of the $\W_n$--invariant subspace of $A^d_n$; this value  is nonnegative and could only grow if $A^*$ were larger. 

The graded $\FIW$--module $V$ is generated by a single generator $x_{1,2}$ in $\FIW$--degree $2$ and graded-degree $\ell$. Then $A^*$ is finitely generated by $x_{1,2}$ as an $\FIW$--algebra, and Lemma \ref{FGCHARALGSTAB} guarantees that the limit $\lim_{n \to \infty} \langle P, A^d_n \rangle_{\W_n}$ exists.

We first consider the $\FI$--algebra in type $A$. The FI--module $M(\,\Y{2}\,)$ has bases $$ M(\,\Y{2}\,)_n \cong \langle \; x_{i,j} \;  \; | \; \; i \neq j, \quad i,j \in [n], \quad x_{i,j}=x_{j,i} \; \rangle.$$ 
Then $A_n^{d \ell}$ is spanned by monomials in $d$ commuting variables $x_{i,j}$, and we may index these monomials by graphs on vertices labelled by $1, \ldots, n$ and an edge $(i,j)$ for each variable $x_{i,j}$ that occurs. The $S_n$--orbits of these monomials are indexed by unlabelled graphs on $d$ edges. 


For $n$ sufficiently large, say $n>2d$, the number of these unlabelled graphs is independent of $n$; we can simply consider those graphs with $d$ edges and without isolated vertices. We will, in fact, only need to consider monomials without repeated variables, so the corresponding graphs have no multi-edges. 

Let $G_d$ denote the number of graphs on $d$ edges without loops, multi-edges, or isolated vertices. Lupanov \cite{LupanovEdges} showed that $$G_d \gtrsim  \left[\frac{2}{e} \frac{d}{\ln^2 d} \left( \frac{2 \ln(\ln d)}{\ln d}  +1 \right)\right]^d.$$ 
In particular $$ \sqrt[d]{G_d} \sim \frac{2}{e} \frac{d}{\ln^2 d}.$$ 
It follows that the series $ \displaystyle \sum_{d \geq 0} \frac{G_d}{q^{\ell d}}$ does not converge for any $\ell \geq 1$ or  $q>1$. For $n>2d$, $$ \langle 1, A^{\ell d}_n \rangle_{S_n} = \dim_{\k}( (A^{\ell d}_n)^{S_n}) > G_d$$ and we conclude that the series 
$$ \sum_{d=0}^{\infty}  \frac{ \lim_{n \to \infty} \langle P_n, A^d_n \rangle_{S_n}}{q^d} 
= \sum_{d=0}^{\infty} \frac{ \lim_{n \to \infty} \langle P_n, A^{\ell d}_n \rangle_{S_n}}{q^{\ell d}}  $$ 
does not converge for the constant class function $P=1$.

The result in type B/C follows the same argument; again $$ M_{BC}(\,\Y{2}\, , \varnothing)_n \cong \langle x_{i,j} \; | \; i \neq j, \quad i,j \in [n], \quad x_{i,j}=x_{j,i} \rangle$$ with the action of $B_n$ on the generators factoring through the quotient $B_n \twoheadrightarrow S_n$. Thus the $B_n$--invariants are isomorphic to the $S_n$--invariants above, and the same graph-theoretic bounds imply that the series  $$ \sum_{d=0}  \frac{ \lim_{n \to \infty} \langle P_n, A^d_n \rangle_{B_n}}{q^d} $$ will not converge in the case that $P=1$. 
\end{proof}

\begin{rem}  {\bf (Nonconvergence for exterior $\FIW$--algebras on $ M(\,\Y{2}\,)$  or $ M_{BC}(\,\Y{2}\, , \varnothing)$).}  \label{RemarkExterior} We expect that the nonconvergence result of Proposition \ref{FIWAlgebrasOnTrv2} would also hold for the free exterior $\FIW$--algebras on $ M(\,\Y{2}\,)$  or $ M_{BC}(\,\Y{2}\, , \varnothing)$. Heuristically, this is because there is an asymptotic sense in which graphs generically have trivial automorphism groups -- and graphs with trivial automorphism groups will correspond to orbits of anticommutative monomials that are nonzero in the $\W_n$--quotient. 
\end{rem} 

We can leverage the results of Proposition \ref{FIWAlgebrasOnTrv2} to prove that convergence fails for any commutative $\FIW$--algebra generated by a free $\FIW$--module on a $\W_2$--representation.

\begin{thm} {\bf (Nonconvergence for symmetric $\FIW$--algebras generated in $\FIW$--degree 2).}  \label{NOLIMIT}  Let $\k$ be a subfield of $\C$, and $q>1$ an integer. 
 Let $V$ be a graded $\FIW$--module over $\k$ supported in positive grades containing a free $\FIW$--module on a representation of $\W_2$.  Let $A^*$ be an $\FIW$--algebra  containing the free symmetric algebra on $V$. Then there exist character polynomials $P$ for which the following series does not converge:  $$ \sum_{d=0}  \frac{ \lim_{n \to \infty} \langle P_n, A^d_n \rangle_{\W_n}}{q^d}. $$ 
\end{thm} 

\begin{proof} In light of Proposition \ref{FIWAlgebrasOnTrv2}, it suffices to consider the cases that $V$  is equal to $M\left(\,\Y{1,1}\,\right)$ if it is type A, or equal to one of  $$  M_{BC}\left(\,\Y{1,1}\, , \varnothing \right), \quad  M_{BC}(\,\Y{1}\, , \Y{1}), \quad   M_{BC}(\, \varnothing, \; \Y{2}\,), \quad  \text{or} \quad   M_{BC}\left(\, \varnothing, \; \Y{1,1}\,\right)$$ if it is type B/C.

 If $V$ is any of the $\FIW$--modules $M\left(\,\Y{1,1}\,\right)$, $M_{BC}\left(\,\Y{1,1}\, , \varnothing \right)$, $M_{BC}(\, \varnothing, \; \Y{2}\,)$, or $M_{BC}\left(\, \varnothing, \; \Y{1,1}\,\right)$,  then $V_n$ is spanned by elements of the form $x_{i,j}$ with $i \neq j$, $i,j \in [n]$, and  $x_{i,j} = \pm \, x_{j,i}$, where $\W_n$ acts by permuting the indices and negating the variables. Specifically, 
\begin{align*} 
M\left(\,\Y{1,1}\,\right) &= \langle x_{i,j} \; | \; x_{i,j} = -x_{j,i} \rangle && \qquad \\ 
M_{BC}\left(\,\Y{1,1}\, , \varnothing \right) &= \langle x_{i,j} \; | \; x_{i,j} = -x_{j,i} \rangle, && (i \; \oo{i}) \cdot x_{i,j} = x_{i,j}&& \qquad && \qquad\\
M_{BC}(\, \varnothing, \; \Y{2}\,) &= \langle x_{i,j} \; | \; x_{i,j} = x_{j,i} \rangle, && (i \; \oo{i}) \cdot x_{i,j} = -x_{i,j} && \qquad && \qquad\\
M_{BC}\left(\, \varnothing, \; \Y{1,1}\,\right) &= \langle x_{i,j} \; | \; x_{i,j} = -x_{j,i} \rangle, && (i \; \oo{i}) \cdot x_{i,j} = -x_{i,j} && \qquad && \qquad 
\end{align*}
 The squares $(x_{i,j})^2 \in A^*$ span a copy of $M(\,\Y{2}\,)$ in type $A$ or  $M_{BC}(\,\Y{2}\, , \varnothing)$ in type B/C. Since these squares are algebraically independent, $A^*$ contains the polynomial algebra they generate, and the result follows from Proposition \ref{FIWAlgebrasOnTrv2}. 

It remains to address the $\FI_{BC}$--module $M_{BC}(\,\Y{1}\, , \Y{1})$. In $\FI_{BC}$--degree $n$ this module is spanned by variables 
$$\{ x_{i}\otimes y_{j} \; | \; i \neq j, \; i,j \in [n] \}, \qquad   (i \; \oo{i})\cdot x_{i}\otimes y_{j} =  - x_{i}\otimes y_{j}, \quad     (j \; \oo{j})\cdot x_{i}\otimes y_{j} =  x_{i}\otimes y_{j}$$
and an action of $S_n \subseteq B_n$ by permuting the indices. But then the elements $\Big((x_{i}\otimes y_{j})^2 +  ( x_{j}\otimes y_{i})^2 \Big)$ are algebraically independent and span a copy of  $M_{BC}(\,\Y{2}\, , \varnothing) \subseteq A^*$, so again convergence must fail by Proposition \ref{FIWAlgebrasOnTrv2}.  \end{proof}

Theorem \ref{NOLIMIT} suggests the following problem. 

\begin{problem}{\bf (Convergence criteria for $\FIW$--algebra with generators in $\FIW$--degree $\leq 2$).} \label{ProblemConvergenceCriteria}
 Let $A^*$ be a finitely generated $\FIW$--algebra  with generators in $\FIW$--degree $\leq 2$. Given a presentation for $A$ as an $\FIW$--algebra, find combinatorial criteria on the relations that guarantee convergence in the sense of Theorem \ref{ConvergenceCriterion}. 
\end{problem} 

\subsection{The braid arrangement: an example of a convergent algebra} \label{SectionBraid}
 
Church--Ellenberg--Farb \cite[Proposition 4.2]{CEFPointCounting} showed that convergence in the sense of Theorem \ref{ConvergenceCriterion} does hold for the anticommutative $\FI$--algebra generated by $M(\Y{2})$ and subject to the ``Arnold relations". Our Theorem \ref{NOLIMIT} shows that finite generation as an $\FI$--algebra in FI--degree 2 is not enough in general to ensure this form of convergence, and that this result should be viewed as a feature of the relations that define this $\FI$--algebra.


 Let $\k$ be a subfield of $\C$. Consider the {\it Arnold algebra} over $\k$ $$\mathcal A^*_n := \frac{\bigwedge^*[ \a_{i,j} \; | \; i,j \in [n], \;, i\neq j, \; \a_{i,j} = \a_{j,i}]}{ ( \a_{i,j}\a_{j,k} + \a_{j,k}\a_{k,i} + \a_{k,i}\a_{i,j} )} $$ which Arnold \cite{Arnol'd} proved to be isomorphic to the cohomology ring of the pure braid group. This algebra is, equivalently, the cohomology $H^*(\mathcal{M}_{A_{n-1}}(\C);\k)$ of the complex hyperplane complement associated to the symmetric group's reflecting hyperplanes. This hyperplane arrangement is sometimes called the \emph{braid arrangement}.  Church--Ellenberg--Farb show that $\mathcal A^*$ has the structure of an $\FI$-algebra finitely generated by $\mathcal A^1= M(\Y{2})$  \cite[Example 5.1.3]{CEF}.

 Proposition \ref{PropArnoldAlg} below gives a simplified proof of a result of Church--Ellenberg--Farb \cite[Proposition 4.2]{CEFPointCounting}. These authors established the result using an explicit decomposition of each graded piece $\mathcal A^d_n$ as an $S_n$--representation, a decomposition proven in significant work of Lehrer--Solomon \cite[Theorem 4.5]{LehrerSolomon}. Our proof will serve as a warm-up to proving Theorem \ref{TheoremStabilityHyperplanes}, the analogous result in Type B/C.

\begin{prop} {\bf (Convergence for the braid arrangement, {\cite[Proposition 4.2]{CEFPointCounting}}). } \label{PropArnoldAlg} Let $\mathcal A^d_n$ denote the $d$th-graded piece of the  Arnold's algebra. For $q\geq 3$ and any character polynomial $P$, the sum   $$ \sum_{d=0}  \frac{ \lim_{n \to \infty} \langle P_n, \mathcal A^d_n \rangle_{S_n}}{q^d} $$
converges absolutely. 
\end{prop}

 Our approach to the proofs of Proposition \ref{PropArnoldAlg} and Theorem \ref{TheoremStabilityHyperplanes} is driven by our interest in Problem \ref{ProblemConvergenceCriteria}, and suggests the following partial solution: we can establish convergence in these cases because we can reduce the spanning sets for the graded pieces from a set of general labelled graphs to a set of decorated labelled trees. Our proof will use the following lemma on forest enumeration. 

\begin{lem}{\bf (Forest enumeration).} \label{LemForestEnumeration} Let $\mathcal F(d)$ denote the number of unlabelled unrooted forests on $d$ edges. Let $\mathcal F_*(d)$ denote the number of unlabelled forests of rooted trees with the property that roots occur only at leaves. Then $\mathcal F_*(d)$ and $\mathcal F(d)$ are order $O(2.96^d)$.
\end{lem}

\begin{proof}
Otter \cite{OtterTrees} proved that the number of unlabelled trees $t(d)$ on $d$ vertices and $(d-1)$ edges is asymptotically $$ t(d) \sim Cb^d d^{-\frac52}  \qquad \text{with $C \doteq 0.534949606\ldots$ and $b \doteq 2.955765285651994974714818\ldots$}$$ The numerical values for the constants are given by Finch \cite{FinchConstants}; see Flajolet--Sedgewick \cite[VII.5 Equation (58) (p481), or Proposition VII.5 (p475) and Subsection VII.21 (p477)]{FlajoletSedgewick}. 

The number of rooted unlabelled trees $r(d)$ on $d$ vertices and $(d-1)$ edges is asymptotically \begin{align*} r(d) &\sim Db^d d^{-\frac32}  \qquad \text{with $D \doteq 0.43992$ and $b \doteq 2.95576528565\ldots$ as above} \end{align*} See Flajolet--Sedgewick \cite[Figure I.13 (p65)]{FlajoletSedgewick}.

Suppose we have an unlabelled forest on $d$ edges such that each tree has a distinguished leaf. We can identify the distinguished leaves to a single vertex to create a rooted tree. This operation is invertible; given a rooted tree we can recover the forest of rooted trees by totally disconnecting all edges incident on the root. Hence $\mathcal F_*(d)=r(d+1)$ is asymptotically given by $$ r(d+1) \sim Db^{d+1} (d+1)^{-\frac32}  < (b)b^d < 3b^d.$$
This proves the bound on $\mathcal F_*(d)$. Since there are more forests of rooted than unrooted trees, this also gives an asymptotic upper bound on $\mathcal F(d)$.
\end{proof}

\begin{proof}[Proof of Proposition \ref{PropArnoldAlg}]
By Lemma \ref{FGCHARALGSTAB}, the limit in the numerator $\lim_{n \to \infty} \langle P_n, \mathcal A^d_n \rangle_{S_n}$ exists. 
		
		To prove that  the sum $$ \sum_{d=0}\frac{ \lim_{n \to \infty} \langle P_n, \mathcal A^d_n \rangle_{S_n}}{q^d} $$ 
		converges absolutely, by Lemma \ref{EQUIV}  it suffices to show that for each $a \geq 0$ there is a function $F_a(d)$ that is independent of $n$ and has order $o(2.99^d)$ so that  $$\dim_{\k}\big((\mathcal A^d_n)^{S_{n-a}}\big) \leq F_a(d) \qquad \text{for all $n$ and $d$.} $$ 

Arnold \cite[Corollary 3]{Arnol'd} proves that an additive basis for $\mathcal A^d_n$ is given by the monomials \begin{align}\label{EquationArnoldBasis} && \a_{i_1, j_i} \wedge \a_{i_2, j_2} \wedge \cdots \wedge \a_{i_d, j_d} \qquad \text{such that} \qquad  i_k < j_k, \text{ and }   j_1 < j_2 < \cdots < j_d.\end{align} 
Other monomials are not necessarily zero, they are merely in the span of those given in Formula (\ref{EquationArnoldBasis}).  Encode each monomial by a graph with vertices labelled by numerals $1, \ldots, n$ and with an edge $(i,j)$ for each factor $\a_{i,j}$. Since the generators $\a_{i,j}$ anticommute, the graph has no multiple edges. The condition on the basis elements (\ref{EquationArnoldBasis}) that the indices increasing imply that their corresponding graphs have no cycles -- that is, each basis element corresponds to a forest. We can therefore take the set of all forests with $d$ edges and vertices labelled by $[n]$ as a (redundant) generating set for $\mathcal A^d_n$.  Recall from Lemma \ref{LemForestEnumeration} that the number of forests $\mathcal F(d)$ with $d$ edges is (quite fortuitously) asymptotically bounded by $3(2.96)^d$.

Consider a forest on $d$ edges with vertices labelled by $[n]$, and an $S_n$-action by permuting the labels. Its $S_{n-a}$-orbit, if nonzero, is encoded by a forest with some vertices assigned distinct labels in $[a]=\{1, 2, \ldots, a\}$ and other vertices unlabelled. We will disregard other additive relations between the monomials of $A^d_n$ corresponding to these forests, and counting these forests' $S_{n-a}$-orbits will  overcount the dimension of $(\mathcal  A^d_n)^{S_{n-a}}$.   The forest has $d$ edges and therefore has at most $2d$ vertices. There are  $2^a$ subsets of $[a]$ and at most $\frac{2d!}{(2d-a)!} $ ways to label vertices of the forest by a given subset. 
Hence the number of orbits of basis elements -- the dimension of $(\mathcal  A^d_n)^{S_{n-a}}$ -- is asymptotically bounded above by  $$ (2^a)\left(\frac{2d!}{(2d-a)!} \right) (3)(2.96)^d. $$ This gives the desired convergence result. 
\end{proof}

The Arnold algebra is of particular interest since, using the Grothendieck--Lefschetz formula, Proposition \ref{PropArnoldAlg} implies a convergence result for certain statistics  on squarefree polynomials in $\F_q[T]$; see Church--Ellenberg--Farb \cite[Theorem 1]{CEFPointCounting}.


\subsection{The algebra of type B/C braid arrangement is convergent} \label{SectionBCBraid}

In this section we consider the analogue of the Arnold algebra in type B/C, and prove the corresponding convergence result in Theorem \ref{TheoremStabilityHyperplanes}. In Section \ref{POLCOUNTING} we will  use this theorem to prove Theorem \ref{ASYMCOUNTING}., a convergence result for certain statistics on squarefree polynomials in $\F_q[T]$.

Let $\k$ be a subfield of $\C$.
Let $\mathcal B^*_n$ denote the  the cohomology algebra of the {\it hyperplane complement of type B/C}  
\begin{align*} \cM_{BC_n}(\C) :&= \Big\{(z_1,z_2,\ldots,z_n)\in\C^n \mid z_i\neq \pm z_j \text{ for }i\neq j; z_i\neq 0\text{ for all } i\Big\},
\end{align*}
the complex hyperplane complement associated to the hyperoctahedral group's reflecting hyperplanes. This hyperplane arrangement is also called the \emph{braid arrangement of type B/C}. As noted in \cite[Theorem 5.8]{FIW2}, $\mathcal B^*$ has the structure of finitely generated $\FI_{BC}$-algebra over $\k$ and  it is generated by the $\FI_{BC}$-module 
$$\mathcal B^1=M_{BC}(\,\Y{1}\, ,\varnothing)\oplus M_{BC}(\,\Y{2}\, ,\varnothing)\oplus M_{BC}(\varnothing,\,\Y{2}\, ).$$

Theorem \ref{NOLIMIT} suggests that finite generation by $\mathcal B^1$ alone should not be enough to ensure the desired convergence properties. Nevertheless, as in type A we will use the algebra's relations to prove Theorem \ref{TheoremStabilityHyperplanes}.


\begin{thm}{\bf (Convergence for the hyperplane arrangement of type B/C)} \label{TheoremStabilityHyperplanes} Let $\mathcal B^d_n$ denote the degree--$d$ graded piece of the cohomology algebra of the hyperplane complement $\cM_{BC_n}(\C)$ of type B/C. For $q\geq 3$ and any hyperoctahedral character polynomial $P$ , the sum   $$ \sum_{d=0}  \frac{ \lim_{n \to \infty} \langle P_n, \mathcal B^d_n \rangle_{B_n}}{q^d}$$
converges absolutely. 
\end{thm}

Before proving Theorem \ref{TheoremStabilityHyperplanes}, we collect some preliminary results. 
The structure of the cohomology ring of the complement $\cM$ of a finite arrangement of complex hyperplanes containing the origin was studied by Brieskorn \cite{Brieskorn} and Orlik--Solomon \cite{OrlikSolomon}.

\begin{thm}(\cite[Theorem 5.2]{OrlikSolomon}; see also \cite[Theorem 2.3 \& Equation (2.2)]{LehrerSolomon}) \label{TheoremOrlikSolomon} 
Let $\cM$ be the complement of a finite complex hyperplane arrangement  $\mathfrak A$. 
Define a set of hyperplanes $H_1, \ldots , H_p$ to be \emph{dependent} if $$ \codim(H_1 \cap \cdots \cap H_p) < p.$$  Then the cohomology ring $H^*(\cM; \C)$ is isomorphic as a graded algebra to the quotient of the exterior algebra $$H^*(\cM; \C) \cong \frac{\bigwedge \langle e_H \; \vert \; H \in \mathfrak A \rangle }{ \langle \quad \sum^p_{\ell=1 } (-1)^\ell \; e_{H_1} \cdots \widehat{e_{H_{\ell}}} \cdots e_{H_p} \quad \vert \quad H_1, \ldots, H_p \text{ dependent} \quad  \rangle }.  $$ 
In particular, any product of dependent hyperplanes $e_{H_1} e_{H_2} \cdots e_{H_p}$ vanishes in the quotient. 
\end{thm}

\subsubsection*{The algebra $\mathcal B^*_n$} By Theorem \ref{TheoremOrlikSolomon}, the cohomology ring $\mathcal B^*_n$ has algebra generators  in bijection with the reflecting hyperplanes of the hyperoctahedral group $B_n$.  Let $\a_{i,j}$, $\b_{i,j}$, and $\g_i$ denote the degree--1 generators corresponding to the hyperplanes $\langle z_i - z_j  \rangle^{\perp}$, $\langle z_i + z_j  \rangle^{\perp}$, and $\langle z_i   \rangle^{\perp}$, respectively, for all $i \neq j$, $i,j \in [n]$.  Observe that $\a_{i,j}=\a_{j,i}$ and $\b_{i,j}=\b_{j,i}$. 

A collection of hyperplanes is dependent precisely when their normal vectors form a linearly dependent set. Observe that the following sets of hyperplanes are dependent in the sense of Theorem \ref{TheoremOrlikSolomon} : \begin{align*} 
\{ \a_{i,j}, \a_{j,k}, \a_{k,i} \},  \{\a_{i,j}, \b_{j,k}, \b_{k,i} \},  \qquad \text{ for distinct } i,j,k \in [n] \\
\{\a_{i,j}, \b_{i,j}, \gamma_i \} , \{\a_{i,j}, \gamma_j, \gamma_i \}, \{ \b_{i,j}, \gamma_j, \gamma_i \}    \qquad \text{ for distinct } i,j \in [n]. 
\end{align*} 
Notably, the sets $\{ \b_{i,j}, \b_{j,k}, \b_{k,i}\}$ and $\{ \a_{i,j}, \a_{j,k}, \b_{k,i}\} $ are not dependent. 
The sets of dependent hyperplanes give the relations 
{\small
\begin{align}
\a_{i,j}\a_{j,k} + \a_{j,k}\a_{k,i} + \a_{k,i}\a_{i,j}, && \a_{i,j}\b_{j,k} + \b_{j,k}\b_{k,i} + \b_{k,i}\a_{i,j},  
\label{EqAB} \\
  \b_{i,j} \gamma_i +\gamma_i\a_{i,j}  + \a_{i,j} \b_{i,j} && \a_{i,j} \gamma_j + \gamma_i\a_{i,j}  +  \gamma_j \gamma_i &&  \b_{i,j} \gamma_j + \gamma_i\b_{i,j}  +  \gamma_j \gamma_i \label{EqGamma}.
\end{align}
}
\begin{lem} \label{LemmaIncreasingIndices} {\bf (Generators for $\mathcal B^*_n$).}
Any word in the elements $\a_{i,j}$ and $\b_{k, \ell}$ is in the span of the words of the form 
$$ \theta_{i_1, j_1}\theta_{i_2, j_2}\theta_{i_3, j_3}\cdots \theta_{i_t, j_t} \g_{i_{t+1}} \cdots \gamma_{i_r}, \qquad \theta_{i_k, j_k} \in \{ \a_{i_k,j_k}, \b_{i_k,j_k} \} ,\qquad j_k < j_{k+1} \text{ and } i_k<j_k.$$
\end{lem}
\begin{proof}
To see this, take any word in $\a_{i,j}$ and $\b_{k, \ell}$. Since $\a_{i,j}=\a_{j,i}$ and $\b_{k, \ell} = \b_{k, \ell}$, we may assume the indices are increasing -- and (up to sign) we may arrange the factors in the word so that second indices are nondecreasing. The goal is to express the word as a linear combination of words with second indices strictly increasing. To accomplish this, we may use the relations in Equation (\ref{EqAB}) inductively replace each factor of the form $\theta_{i,k} \theta_{j,k}$ (with $i<j<k$) by the appropriate expression  \begin{align*}
\a_{i,k}\a_{j,k} =  \a_{i,j}\a_{j,k} - \a_{i,j}\a_{i,k} &&  \a_{i,k}\b_{j,k} =  \b_{i,j}\b_{j,k} - \b_{i,j}\a_{i,k}  \\ 
\b_{i,k}\a_{j,k} =  \b_{i,j}\a_{j,k} - \b_{i,j}\b_{i,k} &&  \b_{i,k}\b_{j,k} =  \a_{i,j}\b_{j,k} - \a_{i,j}\b_{i,k}. \end{align*}
and using the relations in Equation (\ref{EqGamma}) to replace any factor of the form $\alpha_{i,j} \beta_{i,j}$ by 
$$  \a_{i,j} \b_{i,j} =  \a_{i,j}\gamma_i - \b_{i,j}\gamma_i .  \qedhere$$
\end{proof}

\begin{lem} \label{LemmaSpanningTrees} {\bf (Spanning forests for $\mathcal B^d_n$).}
The degree--$d$ graded piece $\mathcal B^d_n$ has a linear spanning set in bijection with decorated forests characterized by the following properties:
\begin{itemize}
\item trees may be rooted or unrooted
\item if a tree is rooted then the root occurs at a leaf
\item nodes are labelled by digits in $[n]$
\item edges are coloured red or blue
\item the number of edges and the number of rooted trees sum to $d$.
\end{itemize}
\end{lem}

\begin{proof}
Consider a word in  $\mathcal B^d_n$. By Lemma \ref{LemmaIncreasingIndices}, we can assume that the subword on the generators $\theta_{i_k, j_k} \in \{ \a_{i_k,j_k}, \b_{i_k,j_k} \}$ satisfies $j_k < j_{k+1}$ and $ i_k<j_k$. Identify this word with a labelled graph by assigning a red edge $(i,j)$ for each factor $\a_{i_k,j_k}$ and a blue edge $(k, \ell)$ for each factor $\b_{k, \ell}$. The condition on the indices guarantees that these graph contains no cycles; this graph is a forest. 

Mark a distinguished vertex $i$ for each factor $\gamma_i$. We claim that there can be at most one distinguished vertex for each connected component of the graph. Suppose otherwise, then the tree must contain a path as in  Figure 1 corresponding to factors $$\g_{i_1} \theta_{i_1, i_2} \theta_{i_2, i_3} \cdots \theta_{i_{t-1}, i_t} \gamma_{i_t}.$$
 
\begin{figure}[h]\label{SelfLoop}
\begin{tikzpicture}[auto, node distance=3cm, every loop/.style={},
                    thick,main node/.style={circle,draw,font=\sffamily\Large\bfseries}]

  \node[main node] (1) {$i_1$};
  \node[main node] (2) [right of=1] {$i_2$};
  \node[main node] (3) [ right of=2] {$i_3$};
  \node[main node] (4) [right of=3] {$i_{t-1}$};
  \node[main node] (5) [right of=4] {$i_{t}$};

  \path[every node/.style={font=\sffamily\small}]
    (1) edge (2)
    (2) edge (3)
(3) edge [dashed] (4)
    (4) edge (5)
     (1) edge [loop left] (1)
(5) edge [loop right] (5) ;
\end{tikzpicture}
\caption{A path corresponding to factors $\g_{i_1} \theta_{i_1, i_2} \theta_{i_2, i_3} \cdots \theta_{i_{t-1}, i_t} \gamma_{i_t}.$  The self-loops represent the $ \gamma_i$ factors.}
\end{figure}
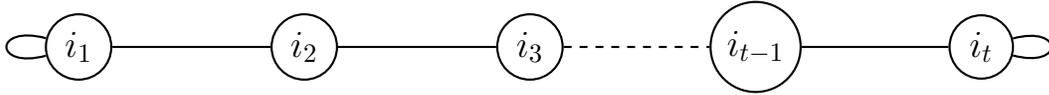

There are $t+1$ corresponding hyperplanes, with normal vectors contained in the $t$-dimensional vector subspace $\langle z_{i_1}, z_{i_2}, \ldots, z_{i_t} \rangle$. These normal vectors must be linearly dependent, and so by Theorem \ref{TheoremOrlikSolomon} the word is equal to zero. It follows that each tree contains at most one distinguished vertex, and we can consider these trees to be rooted. 

It remains to show that we need only those forests where the roots occur at leaves. If we have a factor $\gamma_i$ corresponding to an internal vertex $i$, we can use the relations of Equation (\ref{EqGamma}) 
$$  \gamma_i\a_{i,j} = \a_{i,j} \gamma_j  +  \gamma_i \gamma_j \qquad \text{or} \qquad \gamma_i\b_{i,j} = \gamma_j\b_{i,j}  +  \gamma_i \gamma_j$$
to write our forest as the sum of two forests. The forest associated to the word containing $\gamma_i \gamma_j$ has had the edge $(i,j)$ deleted, and both the vertices $i$ and $j$ are roots of two new trees. In the second forest, the root moved from vertex $i$ across the edge $(i,j)$ to vertex $j$. Iterating this operation, we can 'move' the root from an internal vertex to a leaf of our original forest, and then repeat on the operation on the new set of forests created associated to the terms $\gamma_i \gamma_j$. Each such new forest has one fewer edges than its predecessor, and so this procedure must terminate.  
This concludes the proof. 
\end{proof}

\begin{proof}[Proof of Theorem \ref{TheoremStabilityHyperplanes}]
The graded pieces $\mathcal B^d_n$ of the cohomology ring are finitely generated $\FI_{BC}$--modules \cite[Theorem 5.8]{FIW2}, and so the limit $\lim_{n \to \infty} \langle P_n, \mathcal B^d_n \rangle_{B_n}$ exists by Lemma \ref{FGCHARALGSTAB}. To prove that  the sum $$ \sum_{d=0} \frac{\lim_{n \to \infty} \langle P_n, \mathcal B^d_n \rangle_{B_n}}{q^d} $$ 
		converges absolutely, by Lemma \ref{EQUIV}, it suffices to find a function $F_a(d)$ for each $a \geq 0$ such that $$\dim_{\k}\big((\mathcal B^d_n)^{B_{n-a}}\big) \leq F_a(d) \qquad \text{for all $n$ and $d$.} $$ and such that $F_a(d)$ is independent of $n$ and has order $o(2.99^d)$. By Lemma \ref{LemmaSpanningTrees}, it suffices to show that for each $a$, the number of $B_{n-a}$--orbits of forests described in the Lemma has order $o(2.99^d)$. 

Consider the action of $B_n$ on a forest described in Lemma \ref{LemmaSpanningTrees}. The action of the symmetric group $S_n \subseteq B_n$ permutes the labels $1, 2, \ldots, n$ on the vertices. The action of the transposition $(j \; \overline{j})$ on the vertex $j$ will simultaneously turn all red edges incident to vertex $j$ to blue, and all blue edges to red. As in the proof of Proposition \ref{PropArnoldAlg}, we will disregard any additive relations between the monomials associated to these forests. 

Fix $a$. Vertices $1, 2, \ldots, a$ are labelled, and under the action of $S_{n-a} \subseteq B_{n-a}$ we may consider the remaining vertices in the $B_{n-a}$ quotient to be unlabelled. We claim that it is also possible to consider all but at most $a$ of the edges to be uncoloured.  To see this, first choose a distinguished leaf $v_*$ on each tree, and give each tree a directed graph structure by directing edges away from this leaf. Let $v_e$ denote the terminal vertex of an edge $e$. At most $a$ of the terminal vertices $v_e$ are labelled by a letter in $[a]$. We will show that, under the $B_{n-a}$ action, all colourings are possible for the set of edges $e$ with vertex $v_e$ labelled in $[n]\setminus [a]$, and so we may consider these edges uncoloured in the $B_{n-a}$ quotient.  To obtain a given colouring of the edges labelled in $[n]\setminus [a]$, proceed as follows: first, if the vertex at distance $1$ from $v_*$ is labelled in $[n]\setminus [a]$, act on it by identity or negation to give the edge $e$ the desired colour. Then act on all vertices at distance $2$ from $v_*$ with labels in $[n]\setminus [a]$  to give the associated edges the desired colours. At step $n$, act on all vertices labelled in $[n]\setminus [a]$ at distance $n$ from $v_*$. By construction, the action at step $n$ will not affect any of the earlier colourations, so this procedure will yield the desired colouring. Since all colourings are possible, in the $B_{n-a}$ quotient we may consider these edges to be uncoloured.

We can conclude that $\dim_{\k} ( \mathcal B_n^d)^{B_{n-a}}$ is strictly less than the number of forests of rooted trees on at most $d$ edges with some vertices given distinct labels from a subset of $[a]$ and up to $a$ edges coloured red or blue. Each forest has at most $2d$ vertices, so there are $2^a$ subsets of $[a]$ and at most $\frac{ (2d)!}{(2d-a)!}$ ways to apply these labels.
There are $\sum_{i=0}^a {d \choose i}$ subsets of the $d$ edges of size at most $a$; for $d$ large relative to $a$, $ {d \choose i} \leq {d \choose a}$ so $\sum_{i=0}^a {d \choose i} \leq \sum_{i=0}^a {d \choose a} \leq (a+1) {d \choose a}$. For each subset of edges, there are at most $2^a$ ways to colour these edges either red or blue.

Let $\mathcal F_*(d)$ denote the number of forests where each tree has a distinguished leaf. Putting this together we have, for $d$ large,

\begin{align*} \dim_{\k} ( \mathcal B_n^d)^{B_{n-a}} &\leq  \sum_{i=0}^d \mathcal F_*(i) 2^a\frac{ (2d)!}{(2d-a)!} (a+1) {d \choose a} 2^a \\
&\leq d \mathcal F_*(d) 4^a (a+1) \frac{ (2d)!}{(2d-a)!}  {d \choose a} 
\end{align*}
As a function of $d$, this is $\mathcal F_*(d)$ multiplied by a degree--$(2a+1)$ polynomial in $d$. By Lemma \ref{LemForestEnumeration}, $\mathcal F_*(d)$ is bounded above by  $3(2.96)^d$, and so asymptotically $ \dim_{\k} ( \mathcal B_n^d)^{B_{n-a}} $ is bounded above by
$$ 4^a (a+1) (d) \frac{ (2d)!}{(2d-a)!}   {d \choose a} (3)(2.96)^d, $$
as desired.
\end{proof}

\begin{rem}{\bf (The Church--Ellenberg--Farb approach to Theorem \ref{TheoremStabilityHyperplanes}).}  Theorem \ref{TheoremStabilityHyperplanes} could also be proved using the explicit decomposition of each graded piece $\mathcal B^d_n$ as a $B_n$--representation due to  Douglass (see Theorem \ref{DouglassDecomp} below), and generalizing the argument of Church--Ellenberg--Farb \cite[Proposition 4.2]{CEFPointCounting}. Our approach, however, is more elementary, and highlights the combinatorial features of the algebraic relations in the algebra $\mathcal B^*_n$ that drive this convergence result. 
\end{rem}

\begin{rem} {\bf (Casto's improved bounds).}
In more recent work, Casto \cite[Theorem 3.3]{CASTO} showed that the dimension of the invariant subspaces $\dim_\k((\mathcal  A^d_n)^{S_{n-a}})$ and $\dim_\k((\mathcal  B^d_n)^{S_{n-a}})$ are in fact bounded above by a polynomial in $d$, independently of $n$. His methods apply to a more general setting and are based on a spectral sequence argument.
\end{rem}

\section {Statistics on squarefree polynomials and hyperplane arrangements of type $B/C$}\label{POLCOUNTING}

In this section we use results of Grothendieck, Lehrer, and Kim to describe  the relationship between (on the topological side) the complex cohomology of the hyperplane complement ${\cM_{BC}}_n(\C)$, and (on the combinatorial side) statistics on the set $\cY_n(\F_q)$ of monic squarefree degree-$n$ polynomials in $\F_q[T]$ with nonzero constant term. 

We describe how these statistics capture the distribution of irreducible degree-$r$ factors of  polynomials in $\cY_n(\F_q)$, and the nature of the square roots of their zeroes. By  combining this  formula with  Theorem \ref{TheoremStabilityHyperplanes}, we obtain Theorem \ref{ASYMCOUNTING}, our asymptotic result  about statistics for $\cY_n(\F_q)$. Theorem \ref{ASYMCOUNTING} is the type B/C analogue of Church--Ellenberg--Farb \cite[Theorem 1]{CEFPointCounting}.


\subsection{Hyperplane arrangements of type $B/C$}\label{MBCSchemes}

Consider the canonical action of the hyperoctahedral group $B_n$ on $\C^n$ by signed permutation matrices. Let $\cA_n(\C)$  denote the set of complex codimension one hyperplanes in $\C^n$ fixed by a (complexified) reflection of $B_n$. These hyperplanes are defined by the equations 
\begin{align} \label{HyperplaneEqns}  && z_i-z_j=0 \quad \text{ for all $i \neq j$}, \qquad z_i+z_j=0 \quad \text{ for all $i\neq j$ }, \qquad z_i=0 \quad \text{ for all $i$ }  \end{align} 
with $i,j \in [n]$.    The hyperplane complement of type $B/C$  is the complex manifold given by
$${\cM_{BC}}_n(\C) = \C^n\setminus\bigcup_{H\in\cA_n(\C)} H
= \Big\{(z_1,z_2,\ldots,z_n)\in\C^n: z_i\neq \pm z_j \text{ for }i\neq j; z_i\neq 0\text{ for all } i\Big\}.$$

Since the linear forms that appearing in  (\ref{HyperplaneEqns}) are defined over $\Z$, the complement ${\cM_{BC}}_n(\C)$ can be thought as the set of complex points of a scheme $\cM_n$ over $\Z$. 
More precisely, $$\cM_n=\mathbb{A}_\Z^n\setminus\bigcup_{H\in\cA_n} H,$$ where $\cA_n$ denotes the set of hyperplanes in affine $n$-space $\mathbb{A}_\Z^n$ defined by the linear equations in (\ref{HyperplaneEqns}). Furthermore, the set $\cA_n$  of hyperplanes is stable under the $B_n$--action on $\mathbb{A}_\Z^n$ by signed permutation matrices  and so $\cM_n$ has an action of $B_n$. The quotient scheme $\cY_n:=\cM_n/B_n$  is also defined over $\Z$.

In what follows we let $p>2$ be a fixed prime and let $q=p^r$ for some $r\geq 1$.
Since the  schemes $\cM_n$ and $\cY_n$ are defined over $\Z$, they may be reduced to  schemes over  the finite field $\F_q$. The set ${\cM}_n(\F_q)$   of  $\F_q$--points of ${\cM}_n$  parametrizes $n$--tuples of nonzero points in $\mathbb{A}_{\F_q}^1$ distinct from each other and from their negatives. 
In this chapter we are interested on statistics on the set $\cY_n(\F_q)$ of $\F_q$--points of $\cY_n$. The set  $\cY_n(\F_q)$ may be viewed as a set of polynomials in $\F_q[T]$, as we will now see.

Consider the map $ \psi:\mathbb{A}^n_\mathbb{Z} \longrightarrow \mathbb{A}^n_\mathbb{Z}$ 
 which sends an $n$--tuple $(x_1,\ldots,x_n)\in \mathbb{A}^n_\mathbb{Z}$  to the coefficients of the degree--$n$ polynomial $g(T)=(T-x_1^2)(T-x_2^2)\cdots(T-x_n^2)$. Observe that the space of monic degree--$n$ polynomials is again $\mathbb{A}^n_\mathbb{Z}$  parametrized by the coefficients $(a_0,\ldots, a_{n-1})$ of the polynomials. For each monic polynomial $$g(T)=a_0+a_1 T+\ldots + a_{n-1}T^{n-1}+T^n$$ in the image of $\psi$, 
  the coefficient $a_{n-i}$ is given by a polynomial $f_i(x_1,x_2,\ldots,x_n)$ defined by evaluating the $i^{th}$ elementary symmetric function $e_i$ at the roots  $x_1^2,x_2^2,\ldots, x_n^2$ of $f(T)$,
$$f_i(x_1,x_2,\ldots,x_n):=e_i(x_1^2,x_2^2,\ldots, x_n^2).$$
Specifically,  
$$a_{n-1}=\sum_{1\leq i \leq n} x_i^2, \qquad a_{n-2}=\sum_{1\leq i<j\leq n} x_i^2x_j^2,\qquad \ldots, \qquad  a_0= x_1^2\cdots x_n^2.$$
  Furthermore, there is a natural action of $B_n$ on the affine scheme $\mathbb{A}^n_\mathbb{Z}$ by signed permutation matrices. The map $\psi$ is constant on $B_n$--orbits and it factors through the quotient scheme $(\mathbb{A}^n_\mathbb{Z}/B_n)$.
 It follows form the fundamental theorem of symmetric polynomials that the $B_n$--invariant functions on $\mathbb{A}^n_\mathbb{Z}$ are 
\begin{align*}
\mathbb{Z}[x_1,x_2,\ldots,x_n]^{B_n} & \cong \left( \mathbb{Z}[x_1,x_2,\ldots,x_n]^{(\Z/2\Z)^n} \right)^{B_n/(\Z/2\Z)^n}  \\ 
&\cong\mathbb{Z}[x_1^2,x_2^2,\ldots,x_n^2]^{S_n} \\
&\cong\mathbb{Z}[f_1,f_2,\ldots,f_n] 
\end{align*}
and the map $\psi$ therefore induces an isomorphism of schemes $$(\mathbb{A}^n_\mathbb{Z}/B_n)  \xrightarrow{\cong} \mathbb{A}^n_\mathbb{Z}.$$ 


The restriction of  $\psi$ to ${\cM}_n$  gives us the unramified $B_n$--cover $$\cM_n \rightarrow \cY_n={\cM}_n/B_n.$$  
Since $\bar{\F}_p$ is algebraically closed, the $\bar{\F}_p$--points  $\mathcal{Y}_n(\bar{\F}_p)$ corresponds to the set of $B_n$--orbits of ${\cM}_n(\bar{\F}_p)$. A point in $\mathcal{Y}_n(\F_q)$ corresponds to a set $$\{\pm x_1,\ldots,\pm x_n\}\subset {\bar{\F}_p} \quad \text {with $x_i\neq 0$ for all $i$,  invariant under Gal$({\bar{\F}_p}/{\F_q})$}.$$ 
Given the map $\psi$ and the isomorphism $(\mathbb{A}^n_\mathbb{Z}/B_n)\rightarrow \mathbb{A}^n_\mathbb{Z}$, we interpret the $\F$--points of  $ \cY_n$ as  
$$\mathcal{Y}_n(\F):=\{ f\in\F[T]: \text{ degree-$n$  squarefree with } f(0)\neq0\big\},$$
where $\F$ is the field $\F_q$ or $\bar{\F}_p$.

\subsection{A point-counting formula}   \label{SectionPoint-CountingFormula} 

 Let us consider the set $\mathcal{Y}_n(\bar{\mathbb{F}}_q)$ of $\bar{\mathbb{F}}_q$--points of $\mathcal{Y}_n$. The {\it geometric Frobenius morphism} $$\Fr_q:\mathcal{Y}_n(\bar{\mathbb{F}}_p)\longrightarrow \mathcal{Y}_n(\bar{\mathbb{F}}_p)$$ acts on the coordinates in an affine chart by $x\longmapsto x^q$. The map $\Fr_q$ fixes precisely those points with coordinates in $\mathbb{F}_q$, that is, 
 $$\mathcal{Y}_n(\mathbb{F}_q)=\text{Fix}\big(\Fr_q:\mathcal{Y}_n(\bar{\mathbb{F}}_p)\rightarrow \mathcal{Y}_n(\bar{\mathbb{F}}_p)\big).$$
For each $f\in \mathcal{Y}_n(\mathbb{F}_q)$, the Frobenius morphism $\Fr_q$ fixes the polynomial $f$ and acts on the set $$SQ(f):=\{x\in\bar{\mathbb{F}}_q:f(x^2)=0\}=\{\pm x_1,\pm x_2,\ldots, \pm x_n\}$$ of square roots of the zeroes of $f$. If we were to order the set of roots, then the action of Frobenius would define a signed permutation $\sigma_f \in B_n$.  On the unordered roots the signed permutation $\sigma_f$  is well-defined up to conjugation. Thus  given a class function $\chi$  on $B_n$, we can define $$\chi(f):=\chi(\sigma_f).$$

With this notation we may state the following point-counting formula for the set $\mathcal{Y}_n(\mathbb{F}_q)$:


\begin{thm}\label{COUNTING}{\bf (A point-counting formula for $\mathcal{Y}_n(\mathbb{F}_q)$).} Let $q$ be an integral power of a prime number $p>2$ and let $\chi$ be a class function on $B_n$. Then for each $n\geq 1$ we have
	\begin{align}\label{COUNT}
	&& \sum_{f\in \mathcal{Y}_n(\mathbb{F}_q)} \chi(f)=\sum_{d=0}^n  (-1)^{d}q^{n-d} \big\langle\chi, H^{d}({\cM_{BC_n}} (\mathbb{C});\mathbb{C}\big)\big\rangle_{B_n}.
	\end{align}
\end{thm}

Formula (\ref{COUNT}) relates statistics on the set of polynomials $\mathcal{Y}_n(\mathbb{F}_q)$ with the topology of the hyperplane complement $\cM_{BC_n} (\mathbb{C})$. This result  is the  type B/C analogue to Church--Ellenberg--Farb  \cite[Theorem 1]{CEFPointCounting}. Theorem \ref{COUNTING} follows from Grothendieck's celebrated trace formula, combined with results of Lehrer \cite{LEHRER_LADIC} and Kim \cite{KIM}. Theorem \ref{COUNTING} was independently proven by Matei \cite[Theorem 4]{Matei}.



\subsubsection*{ Grothendieck-Lefschetz formula with twisted coefficients}

The key tool to obtain the point-counting formula (\ref{COUNT}) for $\mathcal{Y}_n(\mathbb{F}_q)$  is the {\em  Grothendieck-Lefschetz fixed-point formula}.  Also called the {\em  Grothendieck's trace formula}, this result is an analogue of the Lefschetz fixed-point theorem for the $\ell$--adic cohomology theory that was developed by Grothendieck and others to resolve the Weil conjectures. 
These influential results  show the   deep connection between the topology of schemes defined over the complex numbers and the arithmetic of schemes defined over finite fields. 

Here we use  a version of the Grothendieck's trace formula  {\em with twisted coefficients}. When applied to the schemes $\cY_n$  the formula yields the following: 

\begin{align}\label{TWISTED}
&& \sum_{f\in \mathcal{Y}_n(\mathbb{F}_q)} \text{tr} (\Fr_q : \mathcal{V}_f)=\sum_{d\geq 0}  (-1)^{d} \text{tr}\big( \Fr_q : H_c^{2n-d}(\mathcal{Y}_{n / \bar{\mathbb{F}}_p};\mathcal{V})\big).
\end{align}
The system of coefficients $\mathcal{V}$ is an $\ell$-adic sheaf over $\mathcal{Y}_n$,
where $\ell$ is prime to $q$.
The left hand side of the formula adds the local contributions of the trace of Frobenius $\Fr_q$ on each stalk $\mathcal{V}_f$ of $\mathcal{V}$ at $f\in \mathcal{Y}_n(\mathbb{F}_q)$.
On the right hand side of the formula  we have the trace of Frobenius on $H_c^{*}(\mathcal{Y}_{n / \bar{\mathbb{F}}_p};\mathcal{V}\big)$,  the compactly-supported $\ell$-adic cohomology of $\mathcal{Y}_{n / \bar{\mathbb{F}}_p}=\cY_n \times_{\F_q}\bar{\F}_p$.

We refer the reader to Deligne's  {\it Rapport sur la formule des traces}  \cite[Th\'eor\`eme 3.2]{DELIGNE} for the general statement and proof of  Formula (\ref{TWISTED}). For an introduction to $\ell$-adic cohomology and its relation with the Weil conjectures see Carter \cite[Chapter 7.1, Appx]{CARTER}  and Hartshorne \cite[Appx C]{HARTSHORNE}. A more detailed exposition can be found in Milne  \cite{MILNE}. Church--Ellenberg--Farb \cite[Section 2]{CEFPointCounting} and  Gadish  \cite{GADISH}  describe a version of the trace formula suited to  applications in representation stability.

\subsubsection*{ A comparison theorem and the action of Frobenius}

 In order to derive  Theorem \ref{COUNTING} from the trace formula (\ref{TWISTED}) we use a comparison theorem that relates the $\ell$-adic cohomology groups of $\cM_{n/\bar{\F}_p}=\cM_n \times_{\F_q}\bar{\F}_p$ and the classical cohomology groups of the complex manifold ${\cM_{BC_n}}(\C)$.  Lehrer  \cite{LEHRER_LADIC} uses the intersection lattice of a hyperplane arrangement  to provide such a comparison result between the singular and $\ell$-adic theories.

The intersection lattice $L(\cA_n)$ of a type $B_n/C_n$ Coxeter arrangement is known to be isomorphic to the {\it signed partition lattice} $\Pi_n^B$. The elements of $\Pi_n^B$ are  partitions of the set $[n]:=\{0,1,\ldots, n\}$ which satisfy the conditions:
\begin{enumerate}
\item[(i)]  any element but the smallest one in each nonzero block can be barred (signed), and
\item[(ii)] the block containing $0$ is called the {\it zero block} and it has no barred elements.  
\end{enumerate} 
A signed partition $\pi\in\Pi_n^B$ encodes an intersection of planes in $\cA_n$ as follows. Take the subspace $\ell_\pi$ in $\mathbb{A}^n$ given by $(x_1,\ldots,x_n)$ defined by the equations
\begin{itemize}
	\item $x_i=x_j$ when $i$ and $j$ are in the same block of $\pi$ and both are barred or both are unbarred
	\item $x_i= -x_j$ if $i$ and $j$ are in the same block of $\pi$ and one is barred and the other unbarred
	\item $x_i=0$ whenever $i$ is in the zero block of $\pi$.
\end{itemize}
For example, the signed partition of the set $[7]$ $$\pi = 0\ 1\ 4\ |\ 2\ \bar{5}\ 7\ |\ 3\ \bar{6}$$ 
corresponds to the linear subspace of $\mathbb{A}^7$ $$\big\{(x_1,x_2,\ldots,x_7) \; |\;   x_2=-x_5=x_7; \ x_3=-x_6;\ x_1=x_4=0\big\}.$$ Conversely, an intersection of hyperplanes in $\cA_n$ determines a signed partition. We order the signed partitions $\Pi_n^B$ by reverse inclusion of the corresponding linear subspaces, and the resultant poset is a lattice; see for example Bjorner--Wachs \cite{LATTICE} for details on this lattice and the isomorphism of lattices $\Pi_n^B \cong L(\cA_n)$.

Let $L(\cA_n)_q$ denote the intersection lattice of hyperplanes reduced modulo $q$. From the above description, it is clear that for any prime $q> 2$ the intersection lattices $L(\cA_n)_q$ is always isomorphic to the intersection lattice  $L(\cA_n)$ over $\C$.   Lehrer's results therefore provide the following  equivariant comparison theorem.

\begin{thm}{\bf (Lehrer's comparison theorem)} \textnormal{\cite[Theorems 1.1 and 1.5]{LEHRER_LADIC}.} \label{COMP}   For any prime $p>2$ and  for each signed permutation $w\in B_n$ 
	
	$$\text{tr}\big(w,H^d({\cM_{BC}}_n(\C);\C)\big)=\text{tr}\big(w,H_c^{2n-d}( \cM_{n/ \bar{\mathbb{F}}_p};\bar{\mathbb{Q}}_\ell)\big).$$
	
	In particular, $$\dim_\C\big(H^d({\cM_{BC}}_n(\C);\C)\big)=\dim_{\bar{\mathbb{Q}}_\ell} \big(H_c^{2n-d}(\cM_{n/ \bar{\mathbb{F}}_p};\bar{\mathbb{Q}}_\ell)\big),$$
	
	
	where $\ell$ is a prime different from $p$.
\end{thm}






Furthermore, Lehrer  also determines the  action of the {\it geometric Frobenius} $\Fr_q$ on the $\ell$-adic cohomology groups of $\cM_{n /\bar{\mathbb{F}}_p}$.

\begin{thm}{\bf (The action of Frobenius on $H_c^{*}(\cM_{n /\bar{\mathbb{F}}_p};\bar{\mathbb{Q}}_\ell)$).} 
\textnormal{(Lehrer \cite[Propostion 2.4]{LEHRER_LADIC}; see also Kim \cite[Theorem 1'] {KIM}).} \label{FROB} 
	Let $p>2$ be a prime, and $q=p^r$ for some $r\geq 1$. Let $\Fr_q$ be the geometric Frobenius morphism $x\mapsto x^q$. Then
	\begin{itemize}
		\item[i)] $H_c^{2n-d}(\cM_{n /\bar{\mathbb{F}}_p};\bar{\mathbb{Q}}_\ell)=0$ unless $d=0,1\ldots, n$.
		\item[ii)] All the eigenvalues of $\Fr_q$ on $H_c^{2n-d}(\cM_{n /\bar{\mathbb{F}}_p};\bar{\mathbb{Q}}_\ell)$ are equal to $q^{n-d}$.
	\end{itemize}
\end{thm}


\begin{proof}[Proof of Theorem \ref{COUNTING}] We use  Theorems \ref{FROB} and  \ref{COMP} above to rewrite  the  Grothendieck's trace formula with twisted coefficients (\ref{TWISTED})  in the form stated  in  (\ref{COUNT}). 
  By the linearity of both sides of the equation,  it suffices to consider the case when $\chi$ is the character of an irreducible $B_n$--representation $V$ over $\mathbb{Q}_\ell$. The	 $B_n$--cover $\cM_n\rightarrow \mathcal{Y}_n$ gives a natural correspondence between the set of finite-dimensional representations $V$ of $B_n$ (up to conjugacy) and  the set of finite-dimensional local systems $\mathcal{V}$ on $\mathcal{Y}_n$  (up to isomorphism) that become trivial when  pulled back to $\cM_n$. 

Every irreducible $B_n$--representation in charateristic zero can be defined over $\mathbb{Z}$. This is true of the irreducible $S_n$--representations, and so it follows for $B_n$ from the construction of the irreducible representations -- a procedure that (up to signs) involves pulling back and inducing up from irreducible representations of $S_n$; see Geck--Pfeiffer \cite{GeckPfeiffer}. Hence
  the  local system $\mathcal{V}$ corresponding to $V$ defines an $\ell$-adic sheaf as required for Formula (\ref{TWISTED}); see for example Gadish \cite[Example 2.3 \& Remark 2.5]{GADISH} for details. Furthermore, each stalk $\mathcal{V}_f$ is isomorphic to the representation $V$, and the Frobenius morphism $\Fr_q$ acts on $\mathcal{V}_f$  as the signed permutation $\sigma_f\in B_n$. Then $$\mathrm{tr}(\Fr_q : \mathcal{V}_f)=\chi(\sigma_f):=\chi(f)$$ and  the left hand sides of Formulas (\ref{TWISTED}) and  (\ref{COUNT}) agree. 
	
	To verify the right hand side of the Formula (\ref{COUNT}) let $\bar{\mathcal{V}}$ be the pullback of $\mathcal{V}$ to $\cM_n$. By a transfer argument 
	$$H_c^{d}(\mathcal{Y}_{n/ \bar{\mathbb{F}}_p};\mathcal{V})\cong H_c^{d}(\cM_{n / \bar{\mathbb{F}}_p};\bar{\mathcal{V}})^{B_n}$$ 
	and since $\bar{\mathcal{V}}$ is trivial over $\cM_n$, we obtain
		$$H_c^{d}(\cM_{n / \bar{\mathbb{F}}_p};\bar{\mathcal{V}})^{B_n}\cong \big(H_c^{d}(\cM_{n / \bar{\mathbb{F}}_p};\mathbb{Q}_\ell)\otimes V\big)^{B_n}.$$
We extend scalars to $\bar{\mathbb{Q}}_\ell$. Theorem \ref{FROB} implies that $\Fr_q$ acts on $H_c^{2n-d}(\cM_{n / \bar{\mathbb{F}}_p};\bar{\mathbb{Q}}_\ell)$ by multiplication by $q^{n-d}$. Since  the $\Fr_q$-action commutes with the $B_n$-action  we have
	$$\text{tr}\big( \Fr_q : H_c^{2n-d}(\mathcal{Y}_{n / \bar{\mathbb{F}}_p};\mathcal{V})\big)=q^{n-d}\dim_{\bar{\mathbb{Q}}_\ell}  \big(H_c^{d}(\cM_{n / \bar{\mathbb{F}}_p};\bar{\mathbb{Q}}_\ell)\otimes V\big)^{B_n}.$$

	Any  representation $V$ of $B_n$ in characteristic zero is self-dual; this follows (for example) from Geck--Pfeiffer \cite[Corollary 3.2.14]{GeckPfeiffer}. Hence   
	$$\dim_{\bar{\mathbb{Q}}_\ell}  \big(H_c^{d}(\cM_{n / \bar{\mathbb{F}}_p};\bar{\mathbb{Q}}_\ell)\otimes V\big)^{B_n} = \big\langle\chi, H_c^{2n-d}(\cM_{n / \bar{\mathbb{F}}_p};\bar{\mathbb{Q}}_\ell)	\big\rangle_{B_n}= \big\langle\chi, H^{d}({\cM_{BC}}_n(\mathbb{C});\mathbb{C})\big\rangle_{B_n},$$
	where the last equality follows from the comparison theorem, Theorem  \ref{COMP}. We have recovered the right-hand side of Formula (\ref{COUNT}), and completed the proof.
		 \end{proof}
	

\subsubsection*{Interpreting the values of character polynomials} 
Consider a squarefree polynomial $f$ with nonzero constant term, and a hyperoctahedral character polynomial $P$. In this subsection we will see how to interpret the value $P(f)$ concretely in terms of the data of the irreducible factors of $f$ and the square roots of their zeroes. To do this we will use the following terminology. 

\begin{defn}{\bf (QR and NQR field elements).} \label{DefnQR} Recall that an element $\theta\in\F_{q}$ is called a {\it quadratic residue over $\mathbb{F}_q$}  if there exists some $x\in\F_q$ such that $x^2=\theta$, otherwise we refer to $\theta$ as a {\it quadratic nonresidue over $\mathbb{F}_q$}. More generally, if $\theta\in\bar{\mathbb{F}}_q$, let $$\deg(\theta):=\min\{r: \theta\in\mathbb{F}_{q^r}\}$$ and let $\sqrt{\theta}$ be one of the solutions of $x^2-\theta$. A nonzero $\theta\in\bar{\mathbb{F}}_q$ will be called a {\it quadratic residue (QR)} if $\deg(\theta)=\deg(\sqrt{\theta})$ and  a {\it quadratic nonresidue (NQR)} otherwise. \end{defn} 

Given an irreducible polynomial $g\in\mathcal{Y}_n(\F_q)$, either all the roots of $g$ are QR, or they all are NQR, which allows for the following classification. 
\begin{defn}{\bf (QR and NQR irreducible polynomials).}
An irreducible polynomial $g(x)$ is called QR if all of its roots are QR, and NQR otherwise. 
\end{defn}
This fact about roots of irreducible polynomials can be proven directly using Galois theory, but is also inherent in the classification of $B_n$ signed cycle types as products of positive and negative cycles (Definition \ref{DefnSignedCycleType}), as we will see below.




For each $f\in \mathcal{Y}_n(\mathbb{F}_q)$, recall that the signed permutation $\sigma_f$ induced by $\Fr_q$ acts on the set $$SQ(f):=\{x\in\bar{\mathbb{F}}_q:f(x^2)=0\}=\{\pm x_1,\pm x_2,\ldots, \pm x_n\}$$ of square roots of the zeroes of $f$.   The image of $\sigma_f$ under the projection $B_n \to S_n$ encodes the action of $\sigma_f$ on the zeroes $\{ x_1^2,x_2^2,\ldots,x_n^2\}$ of $f$. 

For each positive or negative $r$--cycle of  $\sigma_f$, we can consider its corresponding  orbit(s)  $$\{\pm x_{i_1}, \pm x_{i_2},\ldots, \pm x_{i_r}\} \in SQ(f).$$ Recall that, by definition, an $r$--cycle in $B_n$ projects to an $r$--cycle in $S_n$, so the zeroes $\{x^2_{i_1}, x^2_{i_2},\ldots, x^2_{i_r}\}$ form a single orbit under the $\Fr_q$--action. 

\begin{itemize}
	\item Since the morphism $(\Fr_q)^r=\Fr_{q^r}$ fixes the set $\{x_{i_1}^2, x_{i_2}^2,\ldots, x_{i_r}^2\}$ pointwise, then the zeroes $x_{i_1}^2, x_{i_2}^2,\ldots, x_{i_r}^2$ of $f$ are in $\mathbb{F}_{q^r}$. No smaller power of $\Fr_q$ fixes these any of these squares, so they do not lie in any smaller field. This set of roots then corresponds to  an irreducible degree--$r$ factor of $f$. The total number of $r$--cycles in $\sigma_f$ is
		$$X_r(f)+Y_r(f)= \text{\# degree--$r$ irreducible factors  of $f$}.$$
		\item If $\{\pm x_{i_1}, \pm x_{i_2},\ldots, \pm x_{i_r}\}$ corresponds to a positive $r$--cycle of $\sigma_f$, then $(\Fr_q)^r=\Fr_{q^r}$ fixes each $x_i$.  This means that the square roots $x_{i_1}, x_{i_2},\ldots, x_{i_r}$ of the zeroes of  $f$ are in $\mathbb{F}_{q^r}$.   Therefore,
	$$X_r(f)=\text{\# degree--$r$ QR  irreducible factors  of $f$}.$$

		\item If $\{\pm x_{i_1}, \pm x_{i_2},\ldots, \pm x_{i_r}\}$ corresponds to  a negative $r$--cycle of $\sigma_f$, then $(\Fr_q)^r=\Fr_{q^r}$ is a product of $r$ transpositions that interchange the two square roots $\pm x_{i}$ of each square $x_i^2$ in the orbit. Hence the square roots $x_{i_1}, x_{i_2},\ldots, x_{i_r}$ of the zeroes of  $f$ are in $\mathbb{F}_{q^{2r}}$, but not in $\mathbb{F}_{q^{r}}$, and
$$Y_r(f)=\text{\# degree--$r$ NQR irreducible factors  of $f$}.$$
\end{itemize}

We summarize these observations in the following proposition. 

\begin{prop} {\bf (Interpreting character polynomials).}
Let $q$ be an odd prime power. Let $f \in \F_q[T]$ be a squarefree polynomial with nonzero constant term. Let $X_1, Y_1, X_2, Y_2, \ldots$ be the $B_n$ class functions of Definition \ref{DefnHypCharPoly}. Then in the notation of Section \ref{SectionPoint-CountingFormula}, 
$$ X_r(f)+Y_r(f)= \text{\# degree--$r$ irreducible factors  of $f$}, $$
$$ X_r(f)=\text{\# degree--$r$ QR  irreducible factors  of $f$}, $$
$$ Y_r(f)=\text{\# degree--$r$ NQR irreducible factors  of $f$}.$$
\end{prop}

\begin{example}{\bf (Detecting QR and NQR linear factors).} 
Consider the polynomial $$f(T)=T^2-1 \qquad \text{in } \cY_n( \F_7)$$ with roots $x_1=1$ and $x_2=-1$. Let $\epsilon\in\bar{\F}_7$ be a square root of $-1$. Then the Frobenius morphism $\Fr_7$ acts on the set $SQ(f)=\{\pm1, \pm \epsilon \}$ by 
\begin{align*}  1&\longmapsto 1^7=1;  &     \epsilon&\longmapsto \epsilon^7=-\epsilon; \\
 -1 	&\longmapsto (-1)^7=-1;  &  -\epsilon&\longmapsto (-\epsilon)^7=\epsilon,
\end{align*}
  and the signed permutation $\sigma_f$ is $(1 )(\oo{1}) (2 \, \oo{2})\in B_2$.
The cycle type of $\sigma_f$ contains a positive $1$-cycle  $(1 )(\oo{1})$   which corresponds to the QR  linear factor $(T-1)$ and  a negative  $1$-cycle  $(2\, \oo{2})$   which corresponds to the NQR  linear factor $(T+1)$, since $-1$  is not a square in $\F_7$.
\end{example}




\begin{defn}{\bf (Polynomial statistics on $  \bigcup_{n \geq 0} \mathcal{Y}_n(\mathbb{F}_q)$).} We refer to the functions on $ \bigcup_{n \geq 0} \mathcal{Y}_n(\mathbb{F}_q)$ defined by hyperoctahedral character polynomials $P \in \Q[X_1, Y_1, X_2, Y_2, \ldots]$ as \emph{polynomial statistics} on $  \bigcup_{n \geq 0} \mathcal{Y}_n(\mathbb{F}_q)$. 
\end{defn}




\subsection{Asymptotic formula}

We now have all the necessary ingredients to prove our asymptotic result for polynomial statistics on squarefree polynomials over $\F_q$ with nonzero constant term.

\begin{thm}{\bf (Stability for polynomial statistics on $\mathcal{Y}_n(\mathbb{F}_q)$})\label{ASYMCOUNTING} Let $q$ be an integral power of an odd prime. For any polynomial  $P \in \Q[X_1, Y_1, X_2, Y_2, \ldots]$ the normalized statistic $$\frac{\sum_{f\in \mathcal{Y}_n(\mathbb{F}_q)} P(f)}{q^n}$$ converges as $n$ approaches infinity. In fact
		\begin{align}\label{ASYMEQ}
	&&\lim_{n\to \infty} q^{-n}\sum_{f\in \mathcal{Y}_n(\mathbb{F}_q)} P(f)=\sum_{d=0}^\infty  \frac{ \lim_{m\to\infty}\big\langle P_m, H^{d}({\cM_{{BC}_m}(\C)},\C)\big\rangle_{B_m}}{(-q)^{d}}
	\end{align}
	and the series on the right hand side converges.
	
\end{thm}

Theorem \ref{ASYMCOUNTING} states that (appropriately normalized) polynomial statistics on $\mathcal{Y}_n(\mathbb{F}_q)$ converge as $n$ tends to infinity -- and relates the limit to the representation theory of the cohomology groups of the complex hyperplane complements $\cM_{BC_n}(\C)$. 

One consequence of Theorem \ref{ASYMCOUNTING} is that the expected value of polynomial statistics converge. 

\begin{cor}{\bf (Stability for the expected value of  polynomial statistics on $\mathcal{Y}_n(\mathbb{F}_q)$).} \label{CorExpectedValue} Let $q$ be an integral power of an odd prime. For any polynomial  $P \in \Q[X_1, Y_1, X_2, Y_2, \ldots]$ the expected value $$\frac{\sum_{f\in \mathcal{Y}_n(\mathbb{F}_q)} P(f)}{|\cY(F_q)|} $$ of $P$ on $\mathcal{Y}_n(\mathbb{F}_q)$ converges as $n$ tends to infinity, and its limit is
\begin{align*}
	&&\lim_{n\to \infty} \frac{\sum_{f\in \mathcal{Y}_n(\mathbb{F}_q)} P(f)}{|\cY(F_q)|} = \left(\frac{q+1}{q-1}\right) \sum_{d=0}^\infty  \frac{ \lim_{m\to\infty}\big\langle P_m, H^{d}({\cM_{{BC}_m}(\C)},\C)\big\rangle_{B_m}}{(-q)^{d}}.
	\end{align*}
\end{cor}

\begin{proof}
In Proposition \ref{NUMBER} below, we compute $$|\mathcal{Y}_n(\mathbb{F}_q)| =q^n-2q^{n-1}+2q^{n-2}-\ldots+ (-1)^{n-1}2q+(-1)^n.$$ Hence
\begin{align*}
\lim_{n \to \infty} \frac{|\mathcal{Y}_n(\mathbb{F}_q)|}{q^n} 
&= 1 -\frac{2}{q}+\frac{2}{q^2}-\frac{2}{q^3}+\ldots +\frac{(-1)^k(2)}{q^k}+\ldots = \frac{(q-1)}{(q+1)}
\end{align*}
and the result follows from Theorem \ref{ASYMCOUNTING}. 
\end{proof}

We now prove Theorem \ref{ASYMCOUNTING}. 

\begin{proof}[Proof of Theorem \ref{ASYMCOUNTING}]
	We follow the arguments used by Church--Ellenberg--Farb \cite[Theorem 3.13]{CEFPointCounting}. 
We denote  $ H^{d}({\cM_{BC}}_n (\mathbb{C});\mathbb{C})$ by $\mathcal B_n^d$ as before.
In the proof of Theorem  \ref{TheoremStabilityHyperplanes}, we showed that there is a  function $F_P(d)$ that is independent of $n$ and has order $o(2.99^d)$ such that $|\langle P_n, \mathcal B^d_n \rangle_{B_n}|\leq F_P(d)$. Then $$\left\vert \lim_{m\to\infty}\langle P_m, \mathcal B^d_m \rangle_{B_m}\right\vert \leq F_P(d).$$ 
Let $\epsilon>0$. The series $\displaystyle \sum_{d\geq 0} \frac{F_P(d)}{q^d}$ converges absolutely, so  $$\sum_{d\geq I+1} \frac{F_P(d)}{q^d}<\epsilon/2 \qquad \text{ for  some $I\in\mathbb{N}$. }$$ 
Let  $N = 2I+\deg(P)$. The second author \cite[Cor 5.10]{FIW2} proved that the sequence of character $B_n$--representations $\mathcal B^k_n$ are given by a unique character polynomial of degree $\leq 2d$  for all $n$. Then  Proposition \ref{STABPOL} implies that
$$  \lim_{m \to \infty} \langle P_m, \mathcal B^d_m \rangle_{B_m} = \langle P_n, \mathcal B^d_n \rangle_{B_n}\qquad \text {for $d\leq I$ and n$\geq N$}.$$
From Theorem  \ref{TheoremStabilityHyperplanes}, the series $$\sum_{d=0}^{\infty} \frac{\lim_{m \to \infty} \langle P_m, \mathcal B^d_m \rangle_{B_m} }{(-q)^d}$$ converges  absolutely to a limit $L<\infty$. On the other hand, by Theorem \ref{COUNTING},
	$$ q^{-n}\sum_{f\in \mathcal{Y}_n(\mathbb{F}_q)} P(f)= \sum_{d=0}^n  \frac{ \big\langle P_n, \mathcal B_n^d \big\rangle_{B_n}}{(-q)^{d}}.$$
Therefore, if $n\geq N$
\begin{align*}
	\left|L- q^{-n}\sum_{f\in \mathcal{Y}_n(\mathbb{F}_q)} P(f)\right|  &  =  \left|\sum_{d\geq I+1} \frac{ \lim_{m \to \infty} \langle P_m, \mathcal B^d_m \rangle_{B_m} -\langle P_n, \mathcal B^d_n \rangle_{B_n} }{(-q)^d}\right|\\
	& \leq  \sum_{d\geq I+1} \frac{ F_P(d)+F_P(d) }{q^d}< \epsilon. \qedhere
\end{align*} 
\end{proof}

\begin{rem}{\bf(Convergence $\FIW$--CHA).} \label{FIWCHA}  It is possible to define a type B/C analogue of the concept of a ``$\FI$-complement of hyperplane
arrangement'' introduced by Church--Ellenberg--Farb  \cite[Section 3]{CEFPointCounting} in type A.
A type B/C version of \cite[Theorem 3.7]{CEFPointCounting} holds, and our Theorem \ref{COUNTING} is a particular case.  

The cohomology ring of a $\FIW$--CHA has the structure of an $\FIW$-algebra. By Theorem \ref{TheoremOrlikSolomon}, this  $\FIW$--algebra  is finitely generated by the elements in cohomological degree $1$. Since we would insist that the hyperplane arrangement of a $\FIW$--CHA contain the form $(x_1-x_2)$, its degree--1 cohomology must contain  $M(\Y{2})$  (if type A) or  $M_{BC}(\Y{2},\varnothing)$ (if type B/C) as a sub--$\FIW$--module. Theorem \ref{NOLIMIT} and Remark \ref{RemarkExterior} therefore suggest  that finite generation is likely not enough to guarantee that the corresponding point count stabilize as $n$ grows.  

A result analogous to \cite[Theorem 3.13]{CEFPointCounting} follows for $\FIW$--CHA that satisfy the equivalent conditions in Lemma \ref{EQUIV} for a subexponential function $g$. We could call such hyperplane complements {\it convergent $\FIW$-CHA} as in \cite[Definition 3.12]{CEFPointCounting}.
\end{rem}

\subsection{Examples of statistics on squarefree polynomials}\label{EXPLICIT}  

In this section we use Theorems \ref{COUNTING} and \ref{ASYMCOUNTING}  to compute examples of polynomial statistics on $\mathcal{Y}_n(\F_q)$ by analyzing the cohomology groups $H^d(\cM_{BC_n}(\C); \C)$. We emphasize that it is possible to compute these statistics by more direct methods; these computations are not the first, nor necessarily the most efficient, means of obtaining these results; see for example Alegre--Juarez--Pajela \cite{WEIYANSTUDENTS} for a counting method using generating functions. Instead, the computations in this section serve to illustrate the beautiful and unexpected relationships between the combinatorics of squarefree polynomials in $\F_q[T]$ and the representation theory of $H^d(\cM_{BC_n}(\C); \C)$ that follow from the work of Grothendieck, Deligne, and others. 


We begin by reviewing results due to Douglass \cite{DouglassArrangement} on the structure of the cohomology groups $H^d(\cM_{BC_n}(\C); \C)$ of the complex hyperplane complements in type B/C.

\subsubsection*{Douglass' decomposition of $H^d({\cM_{BC}}_n(\C); \C)$}  \label{DouglassTheorem}


A celebrated result of Lehrer--Solomon \cite{LehrerSolomon} gives a decomposition of the cohomology groups $H^d(\cM_{A_n}(\C); \C)$ of the complex hyperplane complements associated to the symmetric groups. Lehrer and Solomon describe these cohomology groups as a sum of certain induced representations of one-dimensional characters of certain $S_n$ subgroups. Douglass \cite{DouglassArrangement} proves a closely analogous result for the cohomology $ H^d(\cM_{BC_n}(\C); \C)$ of the type $B/C$ hyperplane complements: he proved the $B_n$--representation $ H^d(\cM_{BC_n}(\C); \C)$  decomposes  as a sum of certain induced representations of one-dimensional characters $\zeta_{\y}$ of subgroups $\tY_{\y}$ of $B_n$, described below. 

\begin{thm}{\bf  (Douglass' decomposition of $H^d({\cM_{BC}}_n(\C); \C)$).} \label{DouglassDecomp}  \cite[Formula (1.1)]{DouglassArrangement} \\
	For each $d$ with $0 \leq d \leq n$, there is an isomorphism of $\C B_n$--modules 
	$$  H^d({\cM_{BC}}_n(\C); \C) = \bigoplus_{ \substack{ \y = (\y^+, \y^-), \\ \ell(\y^+)=n-d}}\Ind_{\tY_{\y}}^{B_n} \zeta_{\y}. $$ 
\end{thm}

The groups $\tY_{\y}$ and their characters  $\zeta_{\y}$ are defined as follows. 
Consider a double partition $\y=(\y^+, \y^-)$ of $n$ with $$\y^+ = (\y_1, \y_2, \ldots, \y_a) \quad \text{ and } \quad \y^- = (\y_{a+1}, \ldots, \y_{a+b}).$$ Let $B_{\y_1} \times \cdots \times  B_{\y_{a+b}}$ be the corresponding subgroup of $B_n$, where  $B_{\y_i}$ denotes the signed permutation group on the letters $$\Omega_i := \{ \pm( 1+ \sum_{ j<i} \y_j), \ldots, \pm(\y_i +  \sum_{ j<i} \y_j )\}.$$ 
For each factor $B_{\y_i}$, let $x_i$ denote the longest element, that is, the central element that acts by multiplication by $-1$. If $\y_i > 1$, then let $y_i$ denote a positive $\y_i$--cycle in $B_{\y_i}$. Let $\tY_{i} := \langle x_i, y_i \rangle$. Example \ref{SamplePartition} shows these groups in the case that $\y = \left( (3, 1, 1), (2, 2, 1) \right)$.

\begin{example}{\bf (The summand of  $H^7({\cM_{BC}}_{10}(\C); \C)$ indexed by $\y = \left( (3, 1, 1), (2, 2, 1) \right).$)}  \label{SamplePartition} As an example, given the double partition $\y = \left( (3, 1, 1), (2, 2, 1) \right)$ we take the subgroup $$B_3 \times B_1 \times B_1 \times B_2 \times B_2 \times B_1 \subseteq B_{10}.$$ Then the elements $x_i$ and $y_i$ are shown in the following table. 

	\begin{footnotesize}

\begin{align*}
{i} && {\Omega_i} && {x_i} && {y_i} & \qquad \\[.5em] \hline \\ 
1 && \{1, \oo1, 2, \oo2, 3, \oo3 \} && (\oo{1}\; 1)(\oo{2} \; 2)(\oo{3} \; 3) && (1 \; 2 \;3)(\oo{1} \; \oo{2} \; \oo{3}) \\ 
2 && \{4, \oo4 \} &&  (\oo{4} \; 4 ) && \\
3 && \{ 5, \oo5\} && (\oo{5} \; 5) && \\
4 && \{ 6, \oo6, 7, \oo7 \} && (\oo{6} \; 6)(\oo{7} \; 7) && (6 \; 7)(\oo{6} \; \oo{7}) \\ 
5 && \{ 8, \oo8, 9, \oo9 \} && (\oo{8} \; 8)(\oo{9} \; 9) && (8 \; 9)(\oo{8} \; \oo{9}) \\ 
6  && \{ 10, \oo{10} \} && (\oo{10} \; 10) && \\
\end{align*} 

	\end{footnotesize}
	\end{example}

Let $n_r(\y^+)$ denote the number of parts of $\y^+$ of size $r$, and similarly $n_r(\y^-)$. Let $\tH_{\y}$ denote the product of symmetric groups that permutes parts of the same size in each partition: $$ \tH_{\y} := S_{n_{\y_1}(\y^+)} \times \cdots \times S_{n_1(\y^+)} \times  S_{n_{\y_{a+1}}(\y^-)} \times \cdots \times S_{n_1(\y^-)}.$$

{\noindent \bf Example \ref{SamplePartition} continued.} Given $\y = \left( (3, 1, 1), (2, 2, 1) \right)$, we have $$\tH_{\y} = \left\langle (4 \; 5)(\oo{4} \; \oo{5}), (6 \; 8)( \oo{6} \; \oo{8})(7 \;9) (\oo{7} \; \oo{9}) \right\rangle  \cong S_2 \times S_2 .$$

Define $$\tY_{\y} := \tH_{\y} ( \tY_1 \times \cdots \times \tY_{a+b} ).$$ 
We define a linear character $\zeta_{\y}$ of $\tY_{\y} $ by its restrictions to the subgroups  $ \tY_1 \times \cdots \times \tY_{a+b}$ and $ \tH_{\y}$, as follows 
\begin{align*}
\zeta_{\y} :  \tY_1 \times \cdots \times \tY_{a+b} & \longrightarrow \C \\ 
y_i & \longmapsto \eta_{\y_i} := (-1)^{\y_i -1} e^{\frac{2\pi I}{\y_i}}  \qquad \text{ $I$ denotes a root of $-1$ } \\
x_i & \longmapsto 1
\end{align*}

For each symmetric group factor of $\tH_{\y}$, the character acts as either the alternating representation or the trivial representation, depending on whether the factor is permuting parts of $\y^+$ or parts of $\y^-$, and whether the corresponding parts are even or odd. 
\begin{align*}
 \zeta_{\y} : \tH_{\y}  &\longrightarrow \C \\ 
 S_{n_{\y_i}(\y^+)} \ni \s & \longmapsto \mathrm{sign}(\s) && \text{if $\y_i$ is even} \\ 
 S_{n_{\y_i}(\y^+)} \ni \s & \longmapsto 1 && \text{if $\y_i$ is odd} \\ 
 S_{n_{\y_i}(\y^-)} \ni \s  &\longmapsto \mathrm{sign}(\s) && \text{if $\y_i$ is odd} \\ 
 S_{n_{\y_i}(\y^-)} \ni \s  &\longmapsto 1 && \text{if $\y_i$ is even.} \\
\end{align*}

{\noindent \bf Example \ref{SamplePartition} completed.} For the partition $\y = \left( (3, 1, 1), (2, 2, 1) \right)$, the representation $\zeta_{\y}$ is defined as follows. 

	\begin{footnotesize}

\begin{align*}
{x_i} && {y_i} && \tH_{\y} & \qquad \\[.5em] \hline \\ 
 (\oo{1}\; 1)(\oo{2} \; 2)(\oo{3} \; 3) \longmapsto 1 && (1 \; 2 \;3)(\oo{1} \; \oo{2} \; \oo{3}) \longmapsto \eta_3 = e^\frac{2 \pi I}{3} \\ 
  (\oo{4} \; 4 )\longmapsto 1  &&  && (4 \; 5)(\oo{4} \; \oo{5}) \longmapsto 1\\
 (\oo{5} \; 5) \longmapsto 1 && \\
 (\oo{6} \; 6)(\oo{7} \; 7) \longmapsto 1  && (6 \; 7)(\oo{6} \; \oo{7})  \longmapsto \eta_2 = 1 &&  (6 \; 8)( \oo{6} \; \oo{8})(7 \;9) (\oo{7} \; \oo{9}) \longmapsto 1 \\ 
 (\oo{8} \; 8)(\oo{9} \; 9) \longmapsto 1  && (8 \; 9)(\oo{8} \; \oo{9})  \longmapsto \eta_2 = 1\\ 
 (\oo{10} \; 10) \longmapsto 1  && \\
\end{align*} 

	\end{footnotesize}

\subsubsection*{The character of $H^{d} ( {\cM_{BC}}_n (\C) ; \C)$} 

Church--Ellenberg--Farb use Lehrer--Solomon's decomposition of the cohomology of the braid arrangement to compute statistics on the space of monic squarefree polynomials \cite[Propositoin 4.5]{CEFPointCounting}. We can similarly use Douglass's result  to perform computations on $H^{d} ( {\cM_{BC}}_n (\C) ; \C)$. 

Observe that, given a $B_n$--representation $V$ with character $\chi$, 
\begin{align*}
\langle \chi, H^d  ( {\cM_{BC}}_n (\C); \C) \rangle_{B_n} =   \sum_{ \substack{ \y = (\y^+, \y^-), \\ \ell(\y^+)=n-d}} \left\langle \chi, \Ind_{\tY_{\y}}^{B_n} \zeta_{\y} \right\rangle_{B_n} \\ 
=  \sum_{ \substack{ \y = (\y^+, \y^-), \\ \ell(\y^+)=n-d}} \left\langle \Res_{\tY_{\y}}^{B_n} \chi,  \zeta_{\y} \right\rangle_{\tY_{\y}}
\end{align*}
Because the characters $\zeta_{\y}$ are 1-dimensional, computing this inner product is simply a matter of counting the dimension of the subspace of $V$ on which $\tY_{\y}$ acts by $\zeta_{\y}$. We will use this observation to prove Lemmas \ref{X-Y} and \ref{INNER} below.

\subsubsection*{Some statistics on $\mathcal{Y}_n(\mathbb{F}_q)$}

\begin{prop}{\bf (The number of $\mathbb{F}_q$--points in $\mathcal{Y}_n(\mathbb{F}_q).$)}\label{NUMBER} Let $q$ be an odd prime power. The number of degree $n$ monic squarefree polynomials in $\mathbb{F}_q[T]$ with nonzero constant term is 

$$|\mathcal{Y}_n(\mathbb{F}_q)|=q^n-2q^{n-1}+2q^{n-2}-\ldots+ (-1)^{n-1}2q+(-1)^n.$$

\end{prop}
\begin{proof}
If we consider the class function $\chi\equiv 1$,  then Formula (\ref{COUNT}) implies that the number of $\mathbb{F}_q$--points of $\mathcal{Y}_n$ is
$$|\mathcal{Y}_n(\mathbb{F}_q)|=\sum_{d=0}^n  (-1)^{d}q^{n-d} \langle 1, H^d  ( \cM_{BC} ; \C) \rangle_{B_n}.$$
	But
	$$\langle 1, H^d( \cM_{BC} ; \C) \rangle_{B_n} =\dim_\C \big(H^d  ( \mathcal{Y}_n(\C); \C)\big)=b_d;$$
	 the $d$th-Betti number of $\mathcal{Y}_n(\C)$. 
	 Brieskorn \cite[Th\'eor\`eme 7]{Brieskorn} determined these Betti numbers to be $$ b_0 = b_n = 1 \qquad \text{ and } \qquad b_d = 2 \quad \text{for $0 < d < n$,}$$ which gives us the desired formula. 
\end{proof}

There are, of course, methods for counting the polynomials in $\mathcal{Y}_n(\mathbb{F}_q)$ directly. Notably, using Formula (\ref{COUNT}), such a count would give a combinatorial proof of the Betti numbers of $\mathcal{Y}_n(\C)$, and recover Brieskorn's result. 

We next use Douglass' result to find the stable values of the inner products  $\langle \chi, H^d ( \cM_{BC} ; \C) \rangle_{B_n}$ when $\chi$ is given by the character polynomials $X_1-Y_1$ or $X_1+Y_1$. We will use these stable values in Propositions \ref{NUMLINEAR} and \ref{NUMQR} to compute further asymptotic statistics on the spaces $\mathcal{Y}_n(\mathbb{F}_q)$.

\begin{lem}\label{X-Y}  {\bf ($\langle X_1 - Y_1, H^d  ( {\cM_{BC}}_n(\C) ; \C) \rangle_{B_n}$ vanishes).} The inner product $$\langle X_1 - Y_1, H^d  ( {\cM_{BC}}_n(\C) ; \C) \rangle_{B_n} = 0 \qquad \text{ for all $n$  and all $d$. }$$
\end{lem}
\begin{proof}
	The character $X_1 -Y_1$ corresponds to the $B_n$--representation $$V_n \cong V_{\left( (n-1), (1) \right)},  $$ the canonical representation of $B_n$ on $V_n  \cong \C^n$ by signed permutation matrices. To compute the values of the inner product we use the Douglass' result Theorme \ref{DouglassDecomp}. In the notation of Section \ref{DouglassTheorem}, for each partition $\lambda$ the associated product $\prod {x_i}$ is the signed permutation $-$Id. Given a vector $v \in \C^n$, this matrix acts on $v$ by $-1$, while the representation $\zeta_{\y} : \prod {x_i} \to 1$ acts on $v$ trivially, and so $v$ cannot be a copy of the representation $\zeta_{\y}$. We conclude that this inner product is identically zero. 
\end{proof}

\begin{lem}\label{INNER}  {\bf (Stable values of $\left\langle X_1 + Y_1; H^0  ( {\cM_{BC}}_n(\C) ; \C) \right\rangle_{B_n}$).}
The inner product of $ H^d  ( \cM_{BC}(\C) ; \C)$ with the character polynomial $X_1 + Y_1$ has the following stable values 
	\begin{align*}
	\left\langle X_1 + Y_1; H^0  ( {\cM_{BC}}_n(\C) ; \C) \right\rangle_{B_n} &  = 1  \qquad  \text{ for all $n \geq 1$},
	\\ \left\langle X_1 + Y_1; H^d ( {\cM_{BC}}_n(\C) ; \C) \right\rangle_{B_n} & = 4d \qquad  
	\text{ for all $d \geq 1$ and $n \geq d+2$}. 
	\end{align*}\end{lem}

	\begin{proof}
	The characters $ X_1 + Y_1$ correspond to the $B_n$ representation 
	$$V_n  = V_{\left( (n), \varnothing \right)} \oplus  V_{\left( (n-1,1), \varnothing \right)}$$ 
	pulled back from the canonical permutation representation of the symmetric group $S_n$ on $\C^n$ by permuting the $n$ basis elements $e_1, e_2, \ldots, e_n.$  
	Lemma \ref{INNER} can be proved using Douglass' decomposition and (in the notation of Section \ref{DouglassTheorem}) finding for each $\y$ with $\ell(\y^+)=n-d$ the dimension of the subspace of $V_n$ on which $\tY_{\y}$ acts by $\zeta_{\y}$. We summarize this computation for odd homological degrees $d$ in Table \ref{Table-dOdd}. In this table, the notation $\tau_j$ represents the partition 
	$$ \tau_j = \left\{ 
  \begin{array}{l l}
    (2^{\frac{j}{2}}) & \quad \text{if $j$ is even}\\
    (2^{\frac{j-1}{2}},1)  & \quad \text{ if $j$ is odd.}
  \end{array} \right.$$  
	
	\begin{table}[h!]  \caption{Stable values of $\left\langle X_1 + Y_1, H^d ( {\cM_{BC}}_n(\C) ; \C) \right\rangle_{B_n}$ for $d$ odd. The ``action" column shows the representation evaluated on all elements of $\tY_{\y}$ that act nontrivially.}  \begin{center} \label{Table-dOdd}  { \scriptsize
	\begin{tabular}{|c|c|c|c|} \hline 
	&&& \\
		$ \y$ & Basis in $\C^n$ & Action & $ \langle V_n ; \eta_{\y} \rangle_{\tY_{\y}} $ \\ 
		\hline &&& \\
		
		$\left( (1^{(n-d)}), (2^{\frac{d-1}{2}},1) \right)$ & $e_1 +e_2 + \cdots +e_{n-d}$  & Trivial &  Total: 3 \\
		& $e_{n-d+1} +e_2 + \cdots + e_{n-1}$  & Trivial &  \\
		& $e_n$  & Trivial &  \\
		&&& \\
		
		$\left( (2, 1^{(n-d-1)}), (2^{\frac{d-1}{2}}) \right)$ & $e_1 +e_2 $  & Trivial &  Total: 3 \\
		& $e_{3} +e_4 + \cdots + e_{n-d+1}$  & Trivial &  \\
		& $e_{n-d+2} + \cdots + e_{n}$  & Trivial &  \\
		&&& \\
		
		$\left( (j, 1^{(n-d-1)}), \tau_{(d-j+1)} \right)$, & $e_1 + \eta_{j} e_2 + \cdots + \eta_{j}^{j-1} e_{j}$  & $y_1 \longmapsto \eta_{j}$ & Total : $d-1$ \\ 
		$3 \leq j \leq d+1$ &&& \\ &&& \\
		
		$\left( ( j,2,1^{(n-d-2)}), \tau_{(d-j)} \right)$, & $e_1 + \eta_{j} e_2 + \cdots + \eta_{j}^{j-1} e_{j}$  & $y_1 \longmapsto \eta_{j}$ &Total : $d-2$\\ 
		$3 \leq j \leq d$ &&&  \\ &&& \\
		
		$\left( (1^{(n-d)}), (j,  \tau_{(d-j)}) \right)$, & $e_{n-d+1} + \eta_{j} e_{n-d+2} + \cdots + \eta_{j}^{j-1} e_{{n-d+j}}$  & $y_{a+1} \longmapsto \eta_{j}$ & Total : $d-2$ \\ 
		$3 \leq j \leq d$ &&&  \\ &&& \\
		
		$\left( (2,1^{(n-d-1)}), (j,  \tau_{(d-j-1)}) \right)$, & $e_{n-d+1} + \eta_{j} e_{n-d+2} + \cdots + \eta_{j}^{j-1} e_{{n-d+j}}$  & $y_{a+1} \longmapsto \eta_{j}$ & Total : $d-3$ \\ 
		$3 \leq j \leq d-1$ &&&  \\ &&& \\
		
		$\left( (2^2,1^{(n-d-2)}), (2^{\frac{d-3}{2}}, 1) \right)$ & $e_1 +e_2 - e_3 -e_4 $  &  $(1 \; 3)(2 \; 4) \longmapsto -1$  &  Total: 1 \\
		&&& \\
		
		$\left( (2,1^{(n-d-1)}), (2^{\frac{d-3}{2}}, 1^2) \right)$ & $e_{n-1}-e_n $  & $ (n-1 \; n) \longmapsto -1 $ &  Total: 1 \\
		&&& \\
		
		\hline &&& \\ &&& Total: $4d$ \\ \hline		
	\end{tabular}
		} \end{center}
	\end{table} 
	
\noindent  The case of even homological degree $d$ is similar, and we omit the details. In Table \ref{TableX1+Y1} we summarize some stable and unstable values for small $n$ and $d$. 
\end{proof}

\begin{table}[h] 
\begin{footnotesize}	\caption {Some stable and unstable values of $\left\langle X_1 + Y_1, H^d ( {\cM_{BC}}_n(\C) ; \C) \right\rangle_{B_n}$} \label{TableX1+Y1}
	\begin{center}
		\begin{tabular}{| c  |c c c c c c c c c c c | } \hline 
			&& $n$ &&&&&&&&& \\
			$d$ & & 1 & 2 & 3 & 4 & 5 & 6 & 7 & 8 & \ldots & $n >> d$ \\ 
			\cline{2-12}
			&&&&&&&&&&   &\\
			0 &&  1 & 1 & 1 & 1 & 1 & 1 & 1 & 1 & \ldots & 1 \\
			1 &&1 &3 & 4 &  4 & 4 & 4 & 4 & 4  & \ldots & 4 \\ 
			2 && 0  &2 & 6 & 8 & 8 & 8 & 8 & 8 & \ldots & 8 \\ 
			3 && 0 & 0 & 3 & 9 & 12 & 12 & 12 & 12 & \ldots & 12 \\ 
			4 &&  0 & 0 & 0 & 4 & 12 & 16 & 16 & 16 & \ldots & 16 \\ 
			5 &&  0 & 0 & 0 & 0 & 5 & 15 & 20 & 20 & \ldots & 20 \\ 
			6 && 0 & 0 & 0 & 0 & 0 & 6 & 18 & 24 & \ldots & 24 \\ 
			7 && 0 & 0 & 0 & 0 & 0 & 0 & 7 & 21  & \ldots & 28 \\ 
			8 && 0 & 0 & 0 & 0 & 0 & 0 & 0 & 8  & \ldots & 32 \\ 
			&&&&&&&&&&& \\    \hline 
		\end{tabular} \\ \medskip
	\end{center}
\end{footnotesize}
\end{table}

\begin{prop}{\bf (Expected number of linear factors for polynomials in $\mathcal{Y}_n(\mathbb{F}_q)$).} \label{NUMLINEAR} In the limit as $n$ tends to infinity, the expected value of the number of linear factors  in $\mathcal{Y}_n(\mathbb{F}_q)$  converges to 
$$
\lim_{n\to\infty}\frac{\sum_{f\in \mathcal{Y}_n(\mathbb{F}_q)} \big(X_1(f)+Y_1(f)\big)}{|\mathcal{Y}_n(\mathbb{F}_q)|}=\frac{q-1}{q+1}.$$
\end{prop}
\begin{proof}
In order to count the number of linear factors of a given polynomial $f\in\mathcal{Y}_n(\F_q)$ we can evaluate Formula  (\ref{ASYMEQ}) at the character polynomial $X_1+Y_1$.  Lemma \ref{INNER} states the stable values of the inner products on the right-hand side of the equation, and implies
\begin{align*}
\lim_{n\to \infty} q^{-n}\sum_{f\in \mathcal{Y}_n(\mathbb{F}_q)} \big(X_1(f)+Y_1(f)\big)
&=1-\frac{4}{q}+\frac{8}{q^2}-\frac{12}{q^3}+\ldots +\frac{(-1)^k(4k)}{q^k}+\ldots \\ 
& = \frac{(q-1)^2}{(q+1)^2}.
\end{align*}
From Proposition \ref{NUMBER}, we see 
\begin{align*}
\lim_{n \to \infty} \frac{|\mathcal{Y}_n(\mathbb{F}_q)|}{q^n} 
&= 1 -\frac{2}{q}+\frac{2}{q^2}-\frac{2}{q^3}+\ldots +\frac{(-1)^k(2)}{q^k}+\ldots \\ 
&= \frac{(q-1)}{(q+1)}
\end{align*}
and so by taking the ratio of these limits, we complete the proof.
\end{proof}

	
	



Since half the nonzero elements of $\F_q$ are quadratic residues, we might expect that half of these linear factors to be QR in the sense of Definition \ref{DefnQR}  -- and indeed we  can verify this by evaluating Formula (\ref{ASYMEQ}) at the character polynomials $X_1$ and  $Y_1$. 
 Since $$X_1=\frac{1}{2} \bigg[ (X_1+Y_1)+(X_1-Y_1) \bigg] \qquad \text{and} \qquad  Y_1=\frac{1}{2} \bigg[ (X_1+Y_1)-(X_1-Y_1) \bigg] $$ the we can deduce the following result from Lemma \ref{X-Y} and  Theorem \ref{NUMLINEAR}.

\begin{prop}{\bf (Expected number of  QR and NQR linear factors).} \label{NUMQR} The expected number of QR linear factors  in $\mathcal{Y}_n(\mathbb{F}_q)$  tends to
$$
\lim_{n\to\infty}\frac{\sum_{f\in \mathcal{Y}_n(\mathbb{F}_q)} X_1(f)}{|\mathcal{Y}_n(\mathbb{F}_q)|}= \frac{q-1}{2(q+1)}.
$$
The expected number  of NQR linear factors in $\mathcal{Y}_n(\mathbb{F}_q)$  tends to
$$
\lim_{n\to\infty}\frac{\sum_{f\in \mathcal{Y}_n(\mathbb{F}_q)} Y_1(f)}{|\mathcal{Y}_n(\mathbb{F}_q)|}= \frac{q-1}{2(q+1)}.
$$
\end{prop}

\begin{small}

\bibliographystyle{amsalpha}
\bibliography{referFI}\bigskip\bigskip

\noindent Instituto de Matem\'aticas, Universidad Nacional Aut\'onoma de M\'exico \\
Oaxaca de Ju\'arez, Oaxaca, M\'exico 68000 \\ 
rita@im.unam.mx\\

\noindent Stanford University,  Department of Mathematics\\
Sloan Hall.  450 Serra Mall, Building 380,  Stanford, CA 94305\\
jchw@stanford.edu
\end{small}

\end{document}